\documentclass[reqno]{asl}

\usepackage{hyperref}

\title[Diamonds and Dominoes]
    {Diamonds and Dominoes: Impossibility results for associative modal logics}


\author{Søren Brinck Knudstorp}
\revauthor{Knudstorp, Søren Brinck}

\address{ILLC \& Philosophy\\
University of Amsterdam\\
Amsterdam, the Netherlands}

\email{s.b.knudstorp@uva.nl}

\urladdr{https://knudstorp.github.io/}

\newtheorem{Theorem}{Theorem}[section]
\newtheorem{Proposition}[Theorem]{Proposition}
\newtheorem{Lemma}[Theorem]{Lemma}
\newtheorem{Corollary}[Theorem]{Corollary}

\theoremstyle{definition}
\newtheorem{Definition}[Theorem]{Definition}

\newtheorem{Remark}[Theorem]{Remark}

\renewcommand{\qedsymbol}{$\Box$} 
\newcommand{\Sup}{\langle \sup\rangle }
\newcommand{\Inf}{\langle \inf \rangle }
\newcounter{savedcounter} 
\newcounter{savedsection} 

\usepackage{graphicx}
\DeclareMathOperator{\vvee}{\raisebox{0.3ex}{\scalebox{0.6}{$\backslash$}\hspace{-.035cm}\scalebox{0.6}{$\backslash$}\hspace{-.035cm}\scalebox{0.6}{$\slash$}}} 
\newcommand\sepimp{\mathrel{-\mkern-6mu*}} 

\usepackage{tikz} 
\usetikzlibrary{arrows.meta, positioning}

\begin{document}

\maketitle

\begin{abstract}
    We show that for any class of Boolean algebras with an associative operator, if it contains the complex algebra of $(\mathcal{P}(\mathbb{N}), \cup)$, its equational theory is undecidable. Equivalently, any associative normal modal logic valid over the frame $(\mathcal{P}(\mathbb{N}), \cup)$ is undecidable. This settles a long-open question on the decidability of hyperboolean modal logic~\cite{GorankoVakarelov}, and addresses several related problems.
\end{abstract}


\section{Introduction}\label{introduction}


In 1951, Jónsson and Tarski \cite{JonssonTarski1951} introduced the concept of a Boolean algebra with operators (BAO) and developed representation theorems for these structures, marking the first step in the now well-understood connection between relational and algebraic semantics for modal logic (see \cite{bluebook, ChagrovZ97:ml} for detailed treatments).

Fast forward to the 1990s, Kurucz et al. \cite{Kurucz93,KuruczNSS95:jolli} studied the decidability of the equational theory of Boolean algebras with an associative binary operator. Among other results, by a reduction from the Post correspondence problem, they proved that the equational theory of the variety of associative BAOs---corresponding to the least associative normal modal logic $\mathbf{K}_2\oplus (p\circ q)\circ r\leftrightarrow p\circ (q\circ r)$---is undecidable, arguably making it the simplest modal logic known to be undecidable (see \cite{bluebook, ChagrovZ97:ml} for more on decidability and undecidability in modal logic).

The present research lies in continuation with the work of Kurucz and coauthors. By interpreting a domino problem, we prove that any variety of associative BAOs containing the complex algebra of $(\mathcal{P}(\mathbb{N}), \cup)$ has an undecidable equational theory. In modal logical terms, this shows that any normal modal logic extending $\mathbf{K}_2\oplus (p\circ q)\circ r\leftrightarrow p\circ (q\circ r)$ and valid over the frame $(\mathcal{P}(\mathbb{N}), \cup)$ is undecidable. This result carries substantial implications, resolving several open questions, including the long-standing problem concerning the decidability of hyperboolean modal logic, as posed by Goranko and Vakarelov \cite{GorankoVakarelov}.  

The proof techniques developed here may also be of independent interest, having already found broader application. In \cite{Knudstorp2024,Knudstorp:relevant}, similar ideas were applied in relevant logic, showing that the relevant logic $\mathbf{S}$, introduced in \cite{Urquhart72:jsl, Urquhart73}, is undecidable. Likewise, the present proofs inspired recent unpublished joint work with Nick Galatos, Peter Jipsen, and Revantha Ramanayake \cite{GalatosJipsenKnudstorpRevantha:BIund}, establishing the undecidability of Bunched Implication logic (BI) \cite{OHearnP99:bsl}. 

The paper is structured as follows. Section~\ref{sec:defining} defines the relevant logical, algebraic, and relational settings.

Section~\ref{sec:proofmethod} explains the general proof method, which is a reduction from the $\mathbb{N}\times\mathbb{N}$ domino (or tiling) problem. Given a finite set of Wang tiles $\mathcal{W}$---unit squares with a colour on each edge---the problem asks whether it is possible to tile the quadrant $\mathbb{N}\times \mathbb{N}$ so that adjacent tiles match along their shared edges. 

The reduction involves constructing a formula $\phi^{}_\mathcal{W}$ for any such $\mathcal{W}$, and proving the following circular setup:
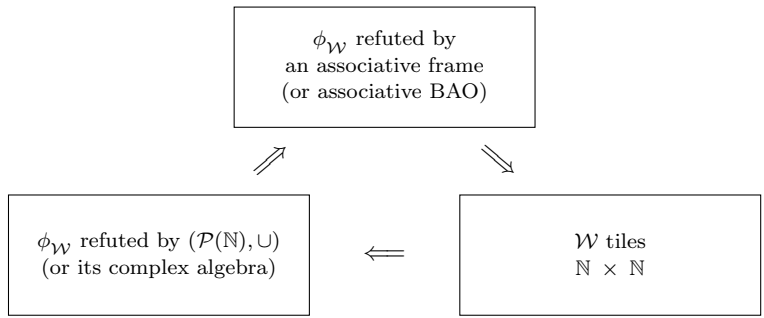
\begin{figure}[ht]

\begin{center}
\begin{tikzpicture}[
  node style/.style={
    font=\footnotesize,
    align=center,
    text width=3.6cm,
    minimum width=4cm,
    minimum height=1.6cm,
    draw,
    rectangle
  },
  arrow style/.style={font=\Large, inner sep=0pt}
]

\node[node style] (A) at (0,2.5)
  {$\phi^{}_\mathcal{W}$ refuted by\\ an associative frame\\ (or associative BAO)};
\node[node style] (B) at (-3,0)
  {$\phi^{}_\mathcal{W}$ refuted by $(\mathcal{P}(\mathbb{N}), \cup)$\\ (or its complex algebra)};
\node[node style] (C) at (3,0)
  {$\mathcal{W}$ tiles\\ $\mathbb{N} \times \mathbb{N}$};

\node at (0,0) {\(\Longleftarrow\)};
\node at (-1.5,1.25) [rotate=39.81] {\(\Longrightarrow\)};
\node at (1.5,1.25) [rotate=-39.81] {\(\Longrightarrow\)};
\end{tikzpicture}
  \label{fig:triangle}
\end{center}
\end{figure}

\noindent From this, and slight modifications, a number of results are derived, including: 
\begin{itemize}
    \item The undecidability of hyperboolean modal logic (Th.~\ref{th:hyp}).
    \item The undecidability of the modal logic of semilattices---or, algebraically, of the equational theory of complex semilattices---(Th.~\ref{th:MIL}), resolving questions raised by Bergman \cite{Bergman}, Jipsen et al. \cite{Jipsen2021}, and the author \cite{Knudstorp2022:thesis,Knudstorp2023:synthese}.
    \item The undecidability of modal logics over lattices studied by Wang and Wang \cite{WangWang2022, WangWang2025} (Th.~\ref{th:lat}).
    \item A new proof of undecidability for the least associative normal modal logic (Th.~\ref{th:associativemodallogic}).
    \item New proofs of undecidability for various versions of propositional separation logic and Boolean Bunched Implication logic (BBI) (Th.~\ref{th:BBI1},~\ref{th:BBI2},~\ref{th:BBI3}).
    \item That there can be no (computable) translation from modal information logic into truthmaker semantics (Th.~\ref{th:notranslation}), answering a question raised by van Benthem \cite{Benthem17:manus, Benthem2024:kripke}.
\end{itemize}

The circular setup follows from two key lemmas, Lemma~\ref{lm:tiling} and~\ref{lm:tilingpowerset}. Their proofs are deferred to Section~\ref{sec:tilinglemmas}, which thereby constitutes the core of the paper.

The main contributions of this paper are its results and proof ideas, yet a more foundational objective of this research is to deepen our understanding of the boundary between the solvable and the unsolvable. To that end, the final section, Section~\ref{sec:features}, compares the achieved undecidability results with known decidability results from related frameworks, including propositional team semantics \cite{YangVäänänen2017}, truthmaker semantics \cite{FineJ19:rsl, Knudstorp2023:synthese}, and modal information logic \cite{Benthem96, Knudstorp2023:jpl}.


\section{Preliminaries}\label{sec:defining}
Algebraically, our setting is varieties of Boolean algebras with an associative operator $\circ$ (associative BAOs), also known as \textit{Boolean semigroups}: Boolean algebras equipped with an associative binary operator $\circ$ that distributes over all finite (including empty) joins. Logically, varieties of associative BAOs correspond to normal extensions of the least associative normal modal logic. 

Formally, we define as follows. 

\begin{Definition}[Logics]
    Let $\mathbf{K}_2$ denote the least normal modal logic over the language $\mathcal{L}$ generated by the grammar
    \begin{align*}
        \varphi\mathrel{::=} p \mid \neg\varphi \mid \varphi \lor \varphi \mid \varphi \circ \varphi,
    \end{align*}
    where $\circ$ is a binary diamond and $p\in \mathsf{Prop}$ for $\mathsf{Prop}$ a countable set of propositional variables. The connectives $\land$, $\to$, $\leftrightarrow$, and the constants $\bot, \top$ are defined as usual. 
    
    By $\mathbf{AK}_2$ we denote the least normal modal logic extending $\mathbf{K}_2$ that contains the associativity axiom for $\circ$: 
        $$(p \circ q) \circ r \leftrightarrow p \circ (q \circ r).$$ 
    That is,
        $$\mathbf{AK}_2\mathrel{:=}\mathbf{K}_2 \oplus (p \circ q) \circ r \leftrightarrow p \circ (q \circ r).$$
    We say that a normal modal logic $\mathbf{L}$ is \textit{associative} if it contains the associativity axiom, i.e., if $\mathbf{L}\supseteq \mathbf{AK}_2$.
\end{Definition}

\begin{Definition}[Algebras and varieties]
    An algebra $\mathfrak{A}=(A, \lor, \land, \neg, \bot, \top, \circ)$ is an \textit{associative BAO} (also referred to as a \textit{Boolean semigroup}) if $(A, \lor, \land, \neg, \bot, \top)$ is a Boolean algebra and $\circ:A\times A\to A$ is an associative binary operator, i.e., the following holds for all $x,y,z\in A$:
\begin{itemize}
    \item Additivity: $x\circ (y\lor z)=(x\circ y)\lor (x\circ z)$ and $(x\lor y)\circ z=(x\circ z)\lor (y\circ z)$
    \item Normality: $x\circ \bot = \bot \circ x=\bot$
    \item Associativity: $(x\circ y)\circ z= x\circ (y\circ z).$
\end{itemize}
Note that additivity and normality amount to distribution of $\circ$ over all finite joins.\footnote{Be aware that we assume normality (distribution over the empty join). 
In some contexts, operators of BAOs are only assumed to distribute over \textit{non-empty} finite joins, and the modifier `normal' is reserved for when distribution extends to the empty join.} We write $\mathsf{BSg}$ for the class of all associative BAOs.

A \textit{variety} of associative BAOs is a class $\mathsf{K}$ of such algebras closed under homomorphic images, subalgebras, and direct products. By $\mathsf{Var}(\mathsf{K})$ we denote the {variety generated by a class} $\mathsf{K}$ of associative BAOs.
\end{Definition}

\textit{Terms} $\mathsf{t}$, \textit{equations} (or \textit{identities}) $\mathsf{t}\approx \mathsf{s}$, and \textit{validity of an equation} $\mathsf{t}\approx \mathsf{s}$ in a class $\mathsf{K}$ of associative BAOs, written $\mathsf{K}\vDash \mathsf{t}\approx \mathsf{s}$, are defined in the standard way. Recall also that, by Birkhoff's theorem, a class of algebras is a variety if and only if it is equationally definable. It follows that the class of all associative BAOs, $\mathsf{BSg}$, is a variety.
\\\\
By algebraic completeness, every associative normal modal logic $\mathbf{L}$ is sound and complete with respect to the variety $\mathsf{V}_\mathbf{L}$ of associative BAOs defined by:
    $$\mathsf{V}_\mathbf{L}\mathrel{:=}\{\mathfrak{A}\mid \mathfrak{A}\vDash \varphi\approx\top, \text{ for all }\varphi\in\mathbf{L}\}.$$
Moreover, the mapping $\mathbf{L}\mapsto \mathsf{V}_\mathbf{L}$ constitutes a dual isomorphism between the lattice of normal extensions of $\mathbf{AK}_2$ and the lattice of subvarieties of $\mathsf{BSg}$, with inverse given by:
    $$\mathsf{V}\mapsto \mathbf{AK}_2\oplus \{\varphi\mid \mathsf{V}\vDash \varphi\approx\top\}.$$
For our purposes, we highlight the following:

\begin{Proposition}\label{pr:interderiveAlgLog}
    Let $\mathbf{L}$ be an associative normal modal logic, $\varphi\in \mathcal{L}$ a formula, and $\mathsf{t}\approx\mathsf{s}$ an equation. Then:
    \begin{itemize}
        \item $\varphi\in\mathbf{L}$ if and only if $\mathsf{V}_\mathbf{L}\vDash \varphi\approx\top$
        \item $\mathsf{V}_\mathbf{L}\vDash \mathsf{t}\approx\mathsf{s}$ if and only if $\mathsf{t}\leftrightarrow \mathsf{s}\in \mathbf{L}$.
    \end{itemize}
    In particular, $\mathbf{L}$ is decidable (i.e., the problem of determining whether $\varphi \in \mathbf{L}$ for $\varphi \in \mathcal{L}$ is decidable) if and only if the equational theory of $\mathsf{V}_\mathbf{L}$ is.
\end{Proposition}
Thus, we have the promised correspondence between (decidability of) associative normal modal logics and (decidability of equational theories of) varieties of associative BAOs.
\\\\
To prove our undecidability results, we proceed via the dual Kripke-semantic perspective. We use that, by canonicity of the associativity axiom (as it is Sahlqvist), the variety of associative BAOs corresponds to the class of all associative Kripke frames. That is, $\mathbf{AK}_2=\mathbf{K}_2\oplus (p\circ q)\circ r\leftrightarrow p\circ (q\circ r)$ is sound and complete with respect to the class of associative Kripke frames, defined as follows:

\begin{Definition}[Frames and models]\label{def:associativeframe}
    An \textit{associative frame} is a pair $\mathbb{F}=(X,R)$ where $R\subseteq X^3$ is an associative ternary relation, i.e., for all $x,a,b,c\in X$,
    \begin{align*}
        Rx(ab)c \text{ iff } Rxa(bc).
    \end{align*}
    Here, $Rx(ab)c$ means that there is $y \in X$ such that $Rxyc$ and $Ryab$, and $Rxa(bc)$ that there is $z \in X$ such that $Rxaz$ and $Rzbc$.\footnote{For a pictorial presentation
    of associativity, and how it relates to tiling and undecidability, see Remark~\ref{rm:AssociativityandTiling}.} In addition to the prefix notation $Rxyz$, we occasionally use the infix $xRyz$ for $(x,y,z)\in R$.
    

    A \textit{model} is a triple $\mathbb{M}=(X,R, V)$ where $(X,R)$ is an associative frame, and $V$ is a \textit{valuation} on $X$, i.e., a function $V:\mathsf{Prop}
\to \mathcal{P}(X)$. 
\end{Definition}

\begin{Remark}\label{rm:operationalperspective}
    Some authors define an associative frame as a pair $\mathbb{F}=(X,\cdot)$ where $\cdot:X\times X\to \mathcal{P}(X)$ is an associative \textit{binary operation}, rather than a ternary relation. This formulation makes the use of the term \textit{associativity} more transparent: associativity of $\cdot$ amounts to the equality
    \begin{align*}
        (x\cdot y)\cdot z=x\cdot (y\cdot z),
    \end{align*}
 where $\cdot$ is lifted to sets $Y,Z\subseteq X$ by
    \begin{align*}
        Y\cdot Z\mathrel{:=}\{x\in X\mid \exists y\in Y, z\in Z:x\in y\cdot z\},\;\; Y\cdot z\mathrel{:=}Y\cdot\{z\},\;\; y\cdot Z\mathrel{:=}\{y\}\cdot Z.
    \end{align*}
    On the other hand, an advantage of the ternary-relation presentation is that it aligns directly with modal logic. 

    In any case, the definitions are equivalent: a ternary relation $R\subseteq X^3$ induces a binary operation $\cdot:X^2\to \mathcal{P}(X)$ by letting $x\in y\cdot z$ iff $Rxyz$. Conversely, a binary operation $\cdot:X^2\to \mathcal{P}(X)$ induces a ternary relation $R\subseteq X^3$ defined by $R\mathrel{:=}\{(x,y,z)\in X^3\mid x\in y\cdot z\}$. The transformations are inverse to each other, and $R\subseteq X^3$ is associative if and only if the corresponding operation $\cdot:X^2\to \mathcal{P}(X)$ is.
\end{Remark}
Note that we have defined associativity from the perspective of the first coordinate; that is, $Rx(ab)c$ holds iff $Rxa(bc)$ does. This choice is also a matter of convention. It contrasts with, e.g., the literature on relevant logic, where associativity is typically defined in terms of the third coordinate: $R(ab)cx$ iff $Ra(bc)x$, but aligns with the usual Kripke semantics for unary diamonds and boxes, which will become clear as we now provide the Kripke-semantic clause for the binary diamond $\circ$.

\begin{Definition}[Frame semantics]\label{def:semantics}
    Given a model $\mathbb{M}=(X,R, V)$ and an $x\in X$, we say that a formula $\varphi\in\mathcal{L}$ is \textit{satisfied} at $x$ (in $\mathbb{M}$), and write $\mathbb{M}, x\Vdash \varphi$ or simply $x\Vdash\varphi$, according to the following recursive clauses:
    \begin{align*}
        &\mathbb{M}, x\Vdash p && \textbf{iff} && x\in V(p),\\
        &\mathbb{M}, x\Vdash \neg\varphi && \textbf{iff} && \mathbb{M}, x\nVdash \varphi,\\
        &\mathbb{M}, x\Vdash \varphi\lor\psi && \textbf{iff} && \mathbb{M}, x\Vdash \varphi \text{ or } \mathbb{M}, x\Vdash \psi,\\
        &\mathbb{M}, x\Vdash \varphi\circ\psi && \textbf{iff} && \text{there exist $y,z\in X$ such that } \mathbb{M}, y\Vdash \varphi; \\
        & && &&\mathbb{M}, z\Vdash \psi; \text{ and $Rxyz$.}
    \end{align*}
    We say that a formula $\varphi\in \mathcal{L}$ is \textit{valid in a frame} $\mathbb{F}=(X,R)$, and write $\mathbb{F}\Vdash \varphi$, if for every model $\mathbb{M}=(X, R, V)$ based on $\mathbb{F}$ and every $x\in X$, we have $\mathbb{M}, x\Vdash \varphi$; and that it is \emph{refuted by} $\mathbb{F}$ if it is not valid in $\mathbb{F}$. It is \textit{valid in a class of frames} $\mathsf{C}$, written $\mathsf{C}\Vdash\varphi$, if it is valid in every frame $\mathbb{F}\in \mathsf{C}$; and we define \textit{the logic of a class of frames} $\mathsf{C}$, $\mathrm{Log}(\mathsf{C})$, as the set of formulas valid in $\mathsf{C}$, i.e., $\mathrm{Log}(\mathsf{C})\mathrel{:=}\{\varphi\in\mathcal{L}\mid \mathsf{C}\Vdash\varphi\}$. 
\end{Definition}
For completeness, denoting the class of associative frames by $\mathsf{A}$, let us make explicit that $\mathbf{AK}_2=\mathrm{Log}(\mathsf{A})$, which, as noted earlier, is a consequence of canonicity of the associativity axiom. 
\begin{Remark}
    If we present models $\mathbb{M}$ from the operational perspective, so $\mathbb{M}=(X,\cdot, V)$, the $\circ$-clause takes the form:
    \begin{align*}
        &\mathbb{M}, x\Vdash \varphi\circ\psi && \textbf{iff} && \text{there exist $y,z\in X$ such that } \mathbb{M}, y\Vdash \varphi; \\
        & && &&\mathbb{M}, z\Vdash \psi; \text{ and $x\in y\cdot z$.}
    \end{align*}
    When $\cdot:X^2\to \mathcal{P}(X)$ is functional (in that $y\cdot z$ always is a singleton), we may consider it a map $\cdot:X^2\to X$ and write $x= y\cdot z$ instead of $x\in y\cdot z$. The induced lifting to sets $\cdot : \mathcal{P}(X) \times \mathcal{P}(X) \to \mathcal{P}(X)$, as defined in Remark~\ref{rm:operationalperspective} and called a \textit{complex operation}, then becomes
    \[
        Y \cdot Z = \{y \cdot z \mid y \in Y, z \in Z\}.
    \]
\end{Remark}
Given an associative frame $\mathbb{F}=(X,R)$, the \textit{complex algebra} of $\mathbb{F}$ is defined as 
\[
    \mathbb{F}^+ = (\mathcal{P}(X), \cup, \cap, {^c}, \varnothing, X, \cdot),
\]
where $\cdot:\mathcal{P}(X)\times\mathcal{P}(X) \to\mathcal{P}(X)$ is the induced lifting, the complex operation, given by
\[
    Y \cdot Z := \{x \in X \mid \text{there are } y \in Y, z \in Z: Rxyz\}.
\]
Observe that $\mathbb{F}^+$ is an associative BAO. For a class of frames $\mathsf{C}$, we write $\mathsf{C}^+$ for the class of complex algebras of frames in $\mathsf{C}$. We also note, for completeness, that for any class of associative frames $\mathsf{C}$, 
\[
    \mathsf{V}_{\mathrm{Log}(\mathsf{C})} = \mathsf{Var}(\mathsf{C}^+).
\]
In particular, $\mathsf{V}_{\mathbf{AK}_2}=\mathsf{Var}(\mathsf{A}^+)=\mathsf{BSg}$, i.e., the variety of associative BAOs, $\mathsf{BSg}$, is generated by the complex algebras of associative frames,  $\mathsf{A}^+$, and precisely contains the algebras validating $\varphi\approx \top$ for every $\varphi\in \mathbf{AK}_2$.
\\\\
With this, we conclude the preliminaries by setting out two central systems; additional systems will be presented alongside the undecidability results in the next section.

First, from the binary-operation perspective $(X, \cdot)$, a natural subclass of associative frames is the class of semilattices, $\mathsf{SL}$, where $\cdot$ is functional (i.e., $\cdot: X^2 \to X$) and associative, commutative, and idempotent. Note that on semilattices, a point $x$ satisfies $\varphi\circ \psi$ if and only if it is the join (or meet) of points $y$ and $z$ satisfying $\varphi$ and $\psi$, respectively. The variety generated by complex algebras of semilattices, $\mathsf{Var}(\mathsf{SL}^+)$, is a variety of associative BAOs and has been studied by Bergman \cite{Bergman}, who, among more, raises the question of its decidability (Problem 4.3)---a question also noted by Jipsen et al. \cite{Jipsen2021}. From the logical perspective, the corresponding logic, $\mathrm{Log}(\mathsf{SL})$, aptly described as the modal logic of semilattices, has been studied by the author in connection with modal information logic and truthmaker semantics \cite{Knudstorp2022:thesis, Knudstorp2023:synthese} (more on this in the next section and again in the final section). There, the same decidability problem was raised, posed in logical terms as whether modal information logic over semilattices is decidable.

Second, the complex algebra of a Boolean algebra $(B, \lor, \land, \neg, \bot, \top)$ is a structure  
    $$(\mathcal{P}(B), \cup, \cap, {^c}, \varnothing, B, \lor, \land, \neg, \bot, \top)$$ 
with two layers of operators: the (\textit{outer} or \textit{external}) Boolean operators $(\cup, \cap, {^c}, \varnothing, B)$, as well as the (\textit{inner} or \textit{internal}) complex operators $(\lor, \land, \neg, \bot, \top)$, where, for $X,Y\in \mathcal{P}(B)$, we have 
    $$X\lor Y\mathrel{:=}\{x\lor y\mid x\in X, y\in Y\},$$
and similarly for the other operators. Observe that the reduct $(\mathcal{P}(B), \cup, \cap, {^c}, \varnothing, B, \lor)$ of a complex Boolean algebra is, in particular, an associative BAO; and similarly that, given a Boolean algebra $(B, \lor, \land, \neg, \bot, \top)$, the join-reduct $(B, \lor)$ (or meet-reduct $(B, \land)$) is an associative \textit{frame}, where, again, relative to a valuation, a point $x\in B$ satisfies $\varphi\circ \psi$ if and only if it is the join (meet) of points $y$ and $z$ satisfying $\varphi$ and $\psi$, respectively. Also, note that $(B, \lor)^+=(\mathcal{P}(B), \cup, \cap, {^c}, \varnothing, B, \lor)$.

Complex Boolean algebras, referred to as \textit{hyperboolean algebras} by Goranko and Vakarelov \cite{GorankoVakarelov}, have a modal analogue in the logic termed \textit{hyperboolean modal logic}, which treats Boolean algebras as Kripke frames and includes modalities corresponding to each Boolean operation, not only the join (or meet). The decidability of this logic was posed as an open problem already in \cite{GorankoVakarelov} and has remained unresolved, with the question raised again by van Benthem \cite{Benthem2024:kripke}, who brought it to the author's attention.

\begin{Remark}
    More recently, Engström and Olsson \cite{EngstromOlsson2023} have examined hyperboolean modal logic from an entirely different perspective, introducing it as a unifying framework for propositional team logics under the name \textit{the Logic of Teams} (LT). This connection to team semantics is particularly relevant for our purposes, as it sheds light on the boundary between the decidable and undecidable---a topic we turn to in the final section. 
\end{Remark}
These are the systems and questions that initially motivated this line of undecidability research, the impact of which, however, reaches further. 


\section{Method and scope}\label{sec:proofmethod}
To establish our undecidability results, we employ a reduction from the tiling problem, introduced by Wang \cite{Wang1963} and shown to be undecidable by Berger \cite{Berger}. The problem is formulated in terms of Wang tiles---unit squares with a colour on each edge---and asks whether, given a finite set of tiles $\mathcal{W}$, it is possible to cover the quadrant $\mathbb{N}\times \mathbb{N}$ so that adjacent tiles match along their shared edges. The formal definition is as follows.

\begin{Definition}[Wang tiling]\label{def:WangTiles}
    A \textit{(Wang) tile} is a 4-tuple $t = (t_u, t_d, t_l, t_r) \in \omega^4$, representing a unit square with edge colours: \( t_u \) (up), \( t_d \) (down), \( t_l \) (left), and \( t_r \) (right).
    
    Given a finite set of tiles $\mathcal{W}$, a \textit{tiling} of the quadrant $\mathbb{N}^2$ is a function $\tau:\mathbb{N}^2\to \mathcal{W}$ such that for all $(m,n)\in \mathbb{N}^2$:
    \begin{itemize}
        \item \( \tau(m,n)_r = \tau(m+1,n)_l \) (adjacent tiles match on horizontal edges),
        \item \( \tau(m,n)_u = \tau(m,n+1)_d \) (adjacent tiles match on vertical edges).
    \end{itemize}
    The \textit{tiling problem} asks whether such a tiling exists for a given finite set $\mathcal{W}$.
\end{Definition}
It was shown to be undecidable by Berger \cite{Berger}.

\begin{Theorem}[\cite{Berger}]\label{th:tilingproblem}
    There is no Turing machine that, given the specifications of an arbitrary finite set of Wang tiles $\mathcal{W}$, decides whether $\mathcal{W}$ tiles $\mathbb{N}^2$.
\end{Theorem}

We reduce the tiling problem to our setting by computably constructing, for each $\mathcal{W}$, a formula $\phi^{}_{\mathcal{W}}\in \mathcal{L}$ such that $\phi^{}_{\mathcal{W}}$ is valid if and only if $\mathcal{W}$ fails to tile the quadrant.

To extend this result across a range of systems, we divide this biimplication into two lemmas, each proving one direction. In the remainder of this section, we state these lemmas, show that they suffice for undecidability, and highlight their consequences. Proofs of these tiling lemmas, along with the definition of the tiling formulas $\phi^{}_{\mathcal{W}}$, are deferred to the next section.

\begin{Lemma}\label{lm:tiling}
    If $\mathcal{W}$ does not tile $\mathbb{N}^2$, then $\phi^{}_\mathcal{W}$ is valid in the least associative normal modal logic (i.e., $\mathsf{BSg}\vDash \phi^{}_\mathcal{W}\approx \top$).
\end{Lemma}

\begin{Lemma}\label{lm:tilingpowerset}
    If $\phi^{}_\mathcal{W}$ is valid in $(\mathcal{P}(\mathbb{N}), \cup)$ (i.e., $(\mathcal{P}(\mathbb{N}), \cup)^+\vDash \phi^{}_\mathcal{W}\approx \top$), then $\mathcal{W}$ does not tile $\mathbb{N}^2$.
\end{Lemma}

Using these two lemmas, we can derive the following theorem, stated both in algebraic and logical terms. We remind the reader that the two formulations are equivalent, and that results in this section---whether stated algebraically or logically---admit an immediate translation via Proposition~\ref{pr:interderiveAlgLog}.
\begin{Theorem}\label{th:undecidable}

            Let $\mathsf{V}$ be a variety of associative BAOs. If $\mathsf{V}$ contains $(\mathcal{P}(\mathbb{N}), \cup)^+$, then its equational theory is undecidable.
            
            Equivalently, let $\mathbf{L}$ be an associative normal modal logic. If $\mathbf{L}\subseteq \textnormal{Log}(\mathcal{P}(\mathbb{N}), \cup)$, then $\mathbf{L}$ is undecidable.
        \end{Theorem}
    
\begin{proof}
    We prove the algebraic version; the logical version then follows by Proposition~\ref{pr:interderiveAlgLog}, although a direct proof is equally simple.
    
    Let $\mathsf{V}$ be a class of algebras such that 
    \[
    (\mathcal{P}(\mathbb{N}), \cup)^+\in \mathsf{V}\subseteq \mathsf{BSg}. 
    \]
    We claim that $\phi^{}_\mathcal{W}\approx \top$ is valid in $\mathsf{V}$ 
    iff $\mathcal{W}$ does not tile the quadrant. Indeed, the left-to-right direction follows from the latter tiling lemma, and right-to-left from the former.
    
    Consequently, the undecidability of the tiling problem implies the undecidability of the equational theory of $\mathsf{V}$. 
\end{proof}
We continue by highlighting the main consequences of this theorem. (The reader interested in the more technical part of the paper may safely skip ahead to Section~\ref{sec:tilinglemmas}---the theorems in this section will all be direct consequences of Theorem~\ref{th:undecidable}, or slight adaptations thereof.) First, the promised solution to the problem posed in \cite{GorankoVakarelov}. 

\begin{Theorem}\label{th:hyp}
            Hyperboolean modal logic is undecidable.
\end{Theorem}
    
\begin{proof}
    Denote by $\mathsf{BA}_\lor$, the class of join-reducts $(B,\lor)$ of Boolean algebras $(B, \lor, \land, \neg, \bot, \top)$. Since the join of a Boolean algebra is associative, and powersets are Boolean algebras, we have $\mathsf{A}\supseteq \mathsf{BA}_\lor\ni (\mathcal{P}(\mathbb{N}), \cup)$, where, we recall, $\mathsf{A}$ is the class of associative frames. Consequently,
    \[
        \mathbf{AK}_2\subseteq \mathrm{Log}(\mathsf{BA}_\lor)\subseteq \textnormal{Log}(\mathcal{P}(\mathbb{N}), \cup),
    \]
    hence 
    $\mathrm{Log}(\mathsf{BA}_\lor)$ is undecidable by Theorem \ref{th:undecidable}. It follows that hyperboolean modal logic 
    is undecidable, as undecidability in a subsignature clearly implies undecidability in the full signature [$\mathrm{Log}(\mathsf{BA}_\lor)$ is, by definition, the $\mathcal{L}$-fragment of hyperboolean modal logic, where our `$\circ$' corresponds to \cite{GorankoVakarelov}'s `$\langle \lor \rangle$'-modality (or, alternatively, their `$\langle \land \rangle$'-modality)]. 
\end{proof}

Second, Theorem~\ref{th:undecidable} also settles the mentioned problem posed in algebraic terms by Bergman \cite{Bergman} and Jipsen et al. \cite{Jipsen2021}, and in logical terminology by the author in \cite{Knudstorp2022:thesis,Knudstorp2023:synthese}.\footnote{A related problem also posed by Bergman \cite{Bergman} and Jipsen et al. \cite{Jipsen2021}, is to find an axiomatisation of $\mathsf{Var}(\mathsf{SL}^+)$. Unbeknownst to the author at the time, coming from the relational perspective of modal information logic, this was solved in the author's Master's thesis \cite{Knudstorp2022:thesis}, in still unpublished work. In further, also unpublished, work \cite{Knudstorp:semilattices}, the author has also shown that it isn't finitely axiomatisable (which solves Problem 4.3 from \cite{Bergman} and a problem left open in \cite{Knudstorp2022:thesis}).}
\begin{Theorem}\label{th:MIL}
    The variety generated by complex algebras of semilattices, $\mathsf{Var}(\mathsf{SL}^+)$, has an undecidable equational theory.

    Equivalently, the modal (information) logic of semilattices, here denoted $\mathrm{Log}(\mathsf{SL})$, is undecidable.
\end{Theorem}
    
\begin{proof}
    Semilattices are associative and $(\mathcal{P}(\mathbb{N}), \cup)$ is, in particular, a semilattice.
\end{proof}
The other main protagonist of \cite{Bergman,Jipsen2021}, apart from $\mathsf{Var}(\mathsf{SL}^+)$,  is the variety $\mathsf{BSl}$ of \textit{Boolean semilattices}---Boolean semigroups that are commutative ($x\circ y\approx y\circ x$) and square-increasing ($x\leq x\circ x$). From Theorem~\ref{th:undecidable}, we derive undecidability of its equational theory, and equivalently of the logic $$\mathbf{K}_2\oplus (p\circ q)\circ r\leftrightarrow p\circ (q\circ r)\oplus (p\circ q)\to(q\circ p)\oplus p\to(p\circ p).$$
\begin{Theorem}
    The variety of Boolean semilattices $\mathsf{BSl}$ has an undecidable equational theory.
\end{Theorem}
\begin{proof}
    Clearly, $\mathsf{BSl}\subseteq \mathsf{BSg}$, and the complex union operation of $(\mathcal{P}(\mathbb{N}), \cup)^+$ is both commutative and square-increasing (in addition to associative).
\end{proof}
Problem 4.4 of \cite{Bergman} asks whether $\mathsf{BSl}$ is generated by its finite members. Since it is finitely axiomatisable, if it were generated by its finite members, it would have a decidable equational theory.
\begin{Corollary}
    $\mathsf{BSl}$ is not generated by its finite members.
\end{Corollary}
Next, though not explicitly raised, we note that it follows from our result that logics recently explored by Wang and Wang \cite{WangWang2022, WangWang2025} are undecidable.
\begin{Theorem}\label{th:lat}
    The modal logics over (modular/distributive) lattices, denoted $\mathbb{HLSI}_L, \mathbb{HLSI}_{ML},$ and $\mathbb{HLSI}_{DL}$ in \cite{WangWang2025}, are all undecidable.
\end{Theorem}
\begin{proof}
    As before, it suffices to note that (i) (modular/distributive) lattices are associative and (ii) powersets are (modular/distributive) lattices. Hence, even the $\mathcal{L}$-fragments of the logics $\mathbb{HLSI}_L, \mathbb{HLSI}_{ML},$ and $\mathbb{HLSI}_{DL}$ are undecidable, where our `$\circ$' corresponds to their `$\Sup$'-modality (or `$\Inf$'-modality, for that matter).
\end{proof}

We should also point out that our proof gives a new, and to the author's knowledge the first tiling-based, proof of several known results. To begin, we have:

\begin{Theorem}[\cite{Kurucz93,KuruczNSS95:jolli}]\label{th:associativemodallogic}
    The least associative normal modal logic, 
        $$\mathbf{AK}_2=\mathbf{K}_2\oplus (p\circ q)\circ r\leftrightarrow p\circ (q\circ r),$$ 
    is undecidable.
\end{Theorem}
Next, let $\mathsf{Rel}$ and $\mathsf{CRel}$ denote the classes of algebras isomorphic to (commutative) algebras of binary relations closed under composition and the Boolean operations of intersection ($\cap$), union ($\cup$), and complementation ($^c$). By a result of Jipsen \cite[Theorem 1]{Jipsen04}, we have $\mathsf{Rel}=\mathsf{Var}(\mathsf{Sg}^+)$ and $\mathsf{CRel}=\mathsf{Var}(\mathsf{CSg}^+)$, where $\mathsf{(C)Sg}$ denotes the variety of (commutative) semigroups. We thus obtain:
\begin{Theorem}[\text{see \cite[Cor. 11.3]{Hirsch2021}\footnotemark}]\footnotetext{Thanks to Peter Jipsen for pointing me to \cite{Hirsch2021} for a proof of the undecidability of the equational theory of $\mathsf{Rel}$. I have not found an earlier source, though I suspect it is not the first: Andr{\'{e}}ka already showed that $\mathsf{Rel}$ is not finitely axiomatisable in \citeyear{Andreka1991} \cite{Andreka1991}. Among related results, the earliest is due to Tarski, who proved the undecidability of the equational theory of representable relation algebras---algebras isomorphic to algebras of binary relations closed under the Boolean operations, composition, and converse, and including the identity relation. For a historical overview and many results on the (un)decidability of equational theories of varieties of relation algebras, see \cite{AndrekaGivantNemeti1997}.}
    $\mathsf{Rel}=\mathsf{Var}(\mathsf{Sg}^+)$ and $\mathsf{CRel}=\mathsf{Var}(\mathsf{CSg}^+)$ have undecidable equational theories.
\end{Theorem}
A final known result that merits emphasis and follows from ours, concerns the undecidability of Boolean Bunched Implication logic (BBI). This was first proven in \cite[Corollary 8.1, stating that $\mathsf{CRM}$ is undecidable]{KuruczNSS95:jolli}, thereby predating the formulation of BBI itself, with (intuitionistic) Bunched Implication logic (BI) only introduced by O'Hearn \& Pym in 1999 \cite{OHearnP99:bsl}. The result remained unknown to the bunched implication community, and BBI's decision problem was long considered open. It was eventually resolved independently by Brotherston \& Kanovich \cite{BrotherstonK10:lics, BrotherstonK14:jacm}, using Minsky machines, and by Larchey-Wendling \& Galmiche \cite{Larchey-WendlingG10:lics, Larchey-WendlingG13:tocl}, embedding an undecidable fragment of intuitionistic linear logic. All of these proofs use distinct approaches, and we supplement these with yet another: tiling.

\begin{Theorem}[\cite{KuruczNSS95:jolli, BrotherstonK10:lics, Larchey-WendlingG10:lics}]\label{th:BBI1}
    Boolean Bunched Implication logic \textnormal{(BBI)} is undecidable.
\end{Theorem}
We also mention that our proof implies the undecidability of several variants of BBI. The paper~\cite{Larchey-WendlingG10:lics} considers three versions of BBI: $\mathbf{BBI}_{\textnormal{ND}}$, $\mathbf{BBI}_{\textnormal{PD}}$, and $\mathbf{BBI}_{\textnormal{TD}}$, with the first being BBI and the latter two stronger systems, and shows them all to be undecidable. These logics are defined as the sets of theorems, or the consequence relations, over classes of structures: non-deterministic (or relational) monoids, partial monoids, and total monoids, respectively.

A \textit{non-deterministic (or relational) monoid} is a triple $(M, \cdot, \epsilon)$, where $\cdot: M \times M \to \mathcal{P}(M)$ is associative and commutative ($x \cdot y = y \cdot x$),\footnote{Note that commutativity is also assumed, even if not reflected in the naming convention. Not that it matters for undecidability, as our result shows.} and $\epsilon\in M$ is a neutral element ($\epsilon\cdot x=\{x\}$). A \textit{partial monoid} is a non-deterministic monoid $(M, \cdot, \epsilon)$ in which $\cdot$ is partial functional, i.e., $x \cdot y$ is either a singleton or empty. A \textit{total monoid} is a partial monoid in which $\cdot$ is (total) functional. Theorem~\ref{th:undecidable} shows that these three logics are undecidable.

\begin{Theorem}[\cite{Larchey-WendlingG10:lics}]\label{th:BBI2}
    $\mathbf{BBI}_{\textnormal{ND}}$, $\mathbf{BBI}_{\textnormal{PD}}$, and $\mathbf{BBI}_{\textnormal{TD}}$ are all undecidable.
\end{Theorem}
\begin{proof}
    Observe that (i) given a non-deterministic monoid $(M, \cdot, \epsilon)$, the pair $(M, \cdot)$ is an associative frame, and (ii) $(\mathcal{P}(\mathbb{N}), \cup, \varnothing)$ is a total monoid. As a result, even the $\mathcal{L}$-fragments of the logics $\mathbf{BBI}_{\textnormal{ND}}$, $\mathbf{BBI}_{\textnormal{PD}}$, and $\mathbf{BBI}_{\textnormal{TD}}$ are undecidable, where our `$\circ$' corresponds to their multiplicative separating conjunction `$\ast$'.
\end{proof}
Like \cite{Larchey-WendlingG10:lics}, in addition to BBI, the paper~\cite{BrotherstonK10:lics} also addresses the decision problem for multiple of its variants. Among more, it considers:
\begin{itemize}
    \item The class of \textit{separation models}, which are triples $(H, \cdot, E)$, where $\cdot :H\times H\to \mathcal{P}(H)$ is partial functional, associative, commutative and cancellative (if $z\cdot x=z\cdot y\neq \varnothing$, then $x=y$), and $E\subseteq H$ is a set of units in that $E\cdot x=\{x\}$.\footnote{In this paper's terminology, a separation model would perhaps be more appropriately referred to as a separation frame.}    
    \item The class of \textit{separation models with indivisible units}, where a separation model $(H, \cdot, E)$ has indivisible units if $x\cdot y\in E$ implies $x\in E$ and $y\in E$.
    \item The class of \textit{CBI-models}, which are quadruples $(H, \cdot, \{e\}, ^{-1})$ such that $(H, \cdot, \{e\})$ is a separation model and $^{-1}:H\to H$ satisfies $x\cdot x^{-1}=e\cdot e^{-1}=e^{-1}$.
    \item The class of \textit{CBI-models with indivisible unit}.
    \item The logic BBI$+$eW, which extends BBI and is included in the logic of all separation models with indivisible units.
    \item The logic CBI, which extends BBI.
    \item The logic CBI$+$eW, which extends CBI and is included in the logic of all CBI-models with indivisible unit.
\end{itemize}
As it stands, the undecidability of these is not an immediate consequence of our main result, since union is not cancellative. However, a simple adaptation of tiling lemma \ref{lm:tilingpowerset} resolves this. Let $(\mathcal{P}(\mathbb{N}), \uplus)$ denote the associative frame where $\uplus$ is the partial disjoint-union operation, defined by $x=y\uplus z$ iff $x=y\cup z$ and $y\cap z=\varnothing$.
\begin{Lemma}
    If $\phi^{}_\mathcal{W}$ is valid in $(\mathcal{P}(\mathbb{N}), \uplus)$ (i.e., $(\mathcal{P}(\mathbb{N}), \uplus)^+\vDash \phi^{}_\mathcal{W}\approx \top$), then $\mathcal{W}$ does not tile $\mathbb{N}^2$.
\end{Lemma}
\begin{proof}
    Reviewing the proof of Lemma \ref{lm:tilingpowerset} in Section~\ref{sec:tilinglemmas}, one sees that the same reasoning applies to $(\mathcal{P}(\mathbb{N}), \uplus)$.
\end{proof}
\begin{Theorem}\label{th:undecidable2}    
            Let $\mathsf{V}$ be a variety of associative BAOs. If $\mathsf{V}$ contains $(\mathcal{P}(\mathbb{N}), \uplus)^+$, then $\mathsf{V}$ is undecidable.

            Equivalently, let $\mathbf{L}$ be an associative normal modal logic. If $\mathbf{L}\subseteq \textnormal{Log}(\mathcal{P}(\mathbb{N}), \uplus)$, then $\mathbf{L}$ is undecidable.
\end{Theorem}
    
\begin{Theorem}[\cite{BrotherstonK10:lics}]\label{th:BBI3}
    The logics \textnormal{BBI$+$eW}, \textnormal{CBI}, and \textnormal{CBI$+$eW}, and the logics of (i.e., the validity problem for) all of the above-listed classes are undecidable.
\end{Theorem}
\begin{proof}
    Clearly, the $\mathcal{L}$-fragment (where `$\circ$' is `$\ast$') of each of these logics extends $\mathbf{AK}_2$. Moreover, $(\mathcal{P}(\mathbb{N}), \uplus, \{\varnothing\}, {^c})$ is a CBI-model with indivisible unit. Hence, the $\mathcal{L}$-fragment of each of the logics is included in $\textnormal{Log}(\mathcal{P}(\mathbb{N}), \uplus)$. Thus, Theorem \ref{th:undecidable2} applies and shows that the logics are undecidable already in their $\mathcal{L}$-fragments.
\end{proof}
As a concluding comment on propositional separation logics and BBI, observe that this not only offers a new proof of the above undecidability results, but also establishes a new result: the $\{\land, \lor, \neg, \ast\}$-fragments of these logics are undecidable. In particular, the logics remain undecidable in the absence of the separating implication (or magic wand) $\sepimp$.
\\\\
On another note, our setting and result connect with Skvortsov's logic of infinite problems \cite{Skvortsov1979}, an analogue of Medvedev's logic of finite problems \cite{Medvedev1962, Medvedev1966}. In their modal versions, these two logics are defined as the sets of theorems, or the consequence relations, in the basic modal language $\{\land, \lor, \neg, \Diamond\}$, interpreted over the classes of Skvortsov and Medvedev frames, respectively. 

A \textit{Skvortsov frame} [resp. \textit{Medvedev frame}] is a pair $(\mathcal{P}(X)\setminus\{\varnothing\}, \supseteq )$, where $X$ is a non-empty [finite] set and `$\supseteq$' denotes the superset relation on $\mathcal{P}(X)\setminus\{\varnothing\}$. The decidability of either logic (whether in its modal or intuitionistic propositional version) remains a long-standing open question.

By considering the union operation on $\mathcal{P}(X)\setminus\{\varnothing\}$, we obtain associative frames $(\mathcal{P}(X)\setminus\{\varnothing\}, \cup)$, where, given a model $\mathbb{M}=(\mathcal{P}(X)\setminus\{\varnothing\}, \cup, V)$, the semantics of $\circ$ becomes:
\begin{align*}
        &\mathbb{M}, x\Vdash \varphi\circ\psi && \text{iff} && \text{there exist $y,z\in \mathcal{P}(X)\setminus\{\varnothing\}$ such that } \mathbb{M}, y\Vdash \varphi; \\
        & && &&\mathbb{M}, z\Vdash \psi; \text{ and $x=y\cup z$.}
\end{align*}
In particular, it is readily verified that 
\begin{align*}
        &\mathbb{M}, x\Vdash \varphi\circ\top && \text{iff} && \text{there exist $y\in \mathcal{P}(X)\setminus\{\varnothing\}$ such that $x\supseteq y$ and } \mathbb{M}, y\Vdash \varphi.
\end{align*}
That is, $\Diamond$ is definable in our language $\mathcal{L}$. Hence, the logic 
    $$\textnormal{Log}\big(\{(\mathcal{P}(X)\setminus\{\varnothing\}, \cup)\mid X\neq \varnothing\}\big)$$ 
in the language $\mathcal{L}$ is a conservative extension of Skvortsov's modal logic.\footnote{Let us also note that a partial converse translation can be defined using that, on Skvortsov models,
    \[
        x\Vdash \Box\Diamond\Box\varphi\circ \Box\Diamond\Box\psi   \quad \text{iff}  \quad x\Vdash \Diamond \Box \varphi\land \Diamond\Box\psi\land\Box\Diamond\Box(\varphi\lor\psi).
    \]
Cf. Theorem \ref{th:undecidable2}, it may also be of interest that if $\circ$ is interpreted w.r.t. the disjoint-union operation $\uplus$ instead, then for $\varphi, \psi$ incompatible (say, $\varphi=\varphi'\land p, \psi=\psi'\land\neg p$), 
    \[
        x\Vdash \Box\varphi\circ \Box\psi  \quad \text{iff}  \quad x\Vdash \Diamond \Box \varphi\land \Diamond\Box\psi\land\Box(\Diamond\Box(\varphi\lor\psi)\land [\Box\Diamond\Box\varphi\to \varphi]\land  [\Box\Diamond\Box\psi\to \psi]).
    \]}

A slight adjustment of tiling lemma \ref{lm:tilingpowerset} yields the undecidability of this extension:
\begin{Lemma}
    If $\phi^{}_\mathcal{W}$ is valid in $(\mathcal{P}(\mathbb{N})\setminus\{\varnothing\}, \cup)$ (i.e., $(\mathcal{P}(\mathbb{N})\setminus\{\varnothing\}, \cup)^+\vDash \phi^{}_\mathcal{W}\approx \top$), then $\mathcal{W}$ does not tile $\mathbb{N}^2$.
\end{Lemma}
\begin{proof}
    The proof of Lemma \ref{lm:tilingpowerset}, see Section~\ref{sec:tilinglemmas}, applies verbatim to $(\mathcal{P}(\mathbb{N})\setminus\{\varnothing\}, \cup)$.
\end{proof}
\begin{Theorem}\label{th:undecidable3}    
            Let $\mathsf{V}$ be a variety of associative BAOs. If $\mathsf{V}$ contains $(\mathcal{P}(\mathbb{N})\setminus\{\varnothing\}, \cup)^+$, then $\mathsf{V}$ is undecidable.

            Equivalently, let $\mathbf{L}$ be an associative normal modal logic. If $\mathbf{L}\subseteq \textnormal{Log}(\mathcal{P}(\mathbb{N})\setminus\{\varnothing\}, \cup)$, then $\mathbf{L}$ is undecidable.
\end{Theorem}
    

\begin{Theorem}
    The conservative extension of Skvortsov's modal logic, 
    \[
        \textnormal{Log}\big(\big\{(\mathcal{P}(X)\setminus\{\varnothing\}, \cup)\mid X\neq \varnothing\big\}\big),
    \] is undecidable.
\end{Theorem}
\begin{proof}
    Frames of the form $(\mathcal{P}(X)\setminus\{\varnothing\}, \cup)$ are associative, and $(\mathcal{P}(\mathbb{N})\setminus\{\varnothing\}, \cup)$ is one such frame.
\end{proof}

Lastly, leveraging one impossibility result to prove another,  we answer a somewhat different open question. Van Benthem \cite{Benthem17:manus, Benthem19:jpl} shows that truthmaker semantics can be faithfully translated into modal information logic over semilattices (see also \cite{Knudstorp2023:synthese} for further work), and in \cite{Benthem17:manus}, he poses the question of whether a converse translational embedding exists---a question reiterated in \cite{Benthem2024:kripke}. We answer this in the negative:

\begin{Theorem}\label{th:notranslation}
    There is no computable function $f$ that maps formulas $\varphi$ in the language of modal information logic to a finite set of premises $\Phi_f$ and a conclusion $\varphi_f$ in the language of truthmaker semantics such that, for all $\varphi$,
    \begin{center}
        $\varphi$ is valid in modal information logic over semilattices iff $\Phi_f$ entails $\varphi_f$ in truthmaker semantics.
    \end{center}
\end{Theorem}

\begin{proof}
    We have proven that modal information logic over semilattices is undecidable (Th. \ref{th:MIL}), and by contrast, truthmaker consequence is not (cf. \cite{FineJ19:rsl,Knudstorp2023:synthese}).
\end{proof}

\section{Proofs of the tiling lemmas}\label{sec:tilinglemmas}
We now define the tiling formula $\phi^{}_{\mathcal{W}}$ for a given set of Wang tiles $\mathcal{W}$ and prove the tiling lemmas, Lemma~\ref{lm:tiling} and~\ref{lm:tilingpowerset}. This forms the core of the paper implying the undecidability results discussed above. 

We begin with a few preliminary definitions and observations that will come in handy. As the tiling lemmas will be proved working frame-theoretically, the following remarks are formulated in that setting, though they may be translated into essentially equivalent algebraic terminology.



\begin{Remark}[Associativity and tiling]\label{rm:AssociativityandTiling}
    Let $(X,R)$ be an associative frame. For $z\in X$, define two binary relations
    \[
        {_zR}\mathrel{:=}\{(a,b)\in X^2\mid Razb \}\quad \text{and} \quad R_z\mathrel{:=}\{(a,b)\in X^2\mid Rabz\}
    \]    
    for `composition' on the left and right, respectively. If $(a,b)\in {_zR}$, we draw an arrow from $a$ to $b$: \tikz[baseline={(a.base)}]{
  \node (a) at (0,0) {$a$};
  \node (b) at (1.5,0) {$b$};
  \draw[->] (a) -- node[midway, below] {${_zR}$} (b);
}. Likewise if $(a,b)\in {R_z}$.  Associativity becomes:

\begin{center}
\hspace{-.235cm}
\begin{tikzpicture}[>=stealth, node distance=2cm]
  \node (a1) at (0,0) {$a$};
  \node (d1) at (0,2) {$d$};
  \node (c1) at (2,2) {$c$};
  \node (implies) at (3.5,1) {$\Rightarrow$};
  
  \node (a2) at (5,0) {$a$};
  \node (b2) at (7,0) {$b$};
  \node (c2) at (7,2) {$c$};
  \node (d2) at (5,2) {$d$};

  \draw[->] (a1) -- node[midway,left]{${R_y}$} (d1);
  \draw[->] (d1) -- node[midway,above]{${_xR}$} (c1);

  \draw[->, dotted] (a2) -- node[midway,below]{${_xR}$} (b2);
  \draw[->, dotted] (b2) -- node[midway,right]{${R_y}$} (c2);

  \draw[->] (a2) -- node[midway,left]{${R_y}$} (d2);
  \draw[->] (d2) -- node[midway,above]{${_xR}$} (c2);
\end{tikzpicture}

\vspace{0.5em}

\small
From $(Rady \text{ and } Rdxc)$ infer $\exists b \in X \, (Raxb \text{ and } Rbcy)$.
\end{center}

and, vice versa,
\begin{center}
\hspace{.235cm}
\begin{tikzpicture}[>=stealth, node distance=2cm]
  \node (a1) at (0,0) {$a$};
  \node (b1) at (2,0) {$b$};
  \node (c1) at (2,2) {$c$};
  \node (implies) at (3.5,1) {$\Rightarrow$};
  
  \node (a2) at (5,0) {$a$};
  \node (b2) at (7,0) {$b$};
  \node (c2) at (7,2) {$c$};
  \node (d2) at (5,2) {$d$};

  \draw[->] (a1) -- node[midway,below]{${_xR}$} (b1);
  \draw[->] (b1) -- node[midway,right]{${R_y}$} (c1);

  \draw[->] (a2) -- node[midway,below]{${_xR}$} (b2);
  \draw[->] (b2) -- node[midway,right]{${R_y}$} (c2);

  \draw[->, dotted] (a2) -- node[midway,left]{${R_y}$} (d2);
  \draw[->, dotted] (d2) -- node[midway,above]{${_xR}$} (c2);
\end{tikzpicture}

\vspace{0.5em}

\small
From $(Raxb \text{ and } Rbcy)$ infer $\exists d \in X \, (Rady \text{ and } Rdxc)$.
\end{center}
As the diagrams illustrate, to encode the quadrant $\mathbb{N}^2$, it suffices to construct an infinite staircase by alternately taking steps along the first axis (via ${_xR}$) and the second axis (via ${R_y}$). The remaining points in the quadrant $\mathbb{N}^2$ are then generated by applying associativity. This idea forms a key component in our definition of the tiling formulas $\phi^{}_\mathcal{W}$ as well as in our proof of Lemma~\ref{lm:tiling}.
\end{Remark}

\begin{Remark}\label{rm:pastlooking}
    We abbreviate 
        \[
            \varphi\hookrightarrow\psi\mathrel{:=}\neg(\varphi\circ\neg\psi),
        \]
    and
        \[
        \psi\hookleftarrow\varphi\mathrel{:=}\neg(\neg\psi\circ\varphi).
        \]
    It is then readily verified that, given a model $\mathbb{M}=(X,R,V)$, the semantics of `$\hookrightarrow$' and `$\hookleftarrow$' come down to the following:
    \begin{align*}
        &\mathbb{M}, x\Vdash \varphi\hookrightarrow\psi && \text{iff} && \text{for all $y,z\in X$ such that } Rxyz: &&\text{if } \mathbb{M}, y\Vdash \varphi, \\
        & && && &&\text{then }\mathbb{M}, z\Vdash \psi.
    \end{align*}
    \begin{align*}
        &\mathbb{M}, x\Vdash \psi\hookleftarrow\varphi && \text{iff} && \text{for all $y,z\in X$ such that } Rxyz: &&\text{if } \mathbb{M}, z\Vdash \varphi, \\
        & && && &&\text{then }\mathbb{M}, y \Vdash \psi.
    \end{align*}
    Additionally, we define a unary operator `$\Box$' as
    \begin{align*}
        \Box\varphi\mathrel{:=} (\top\hookrightarrow  \varphi)\land (\varphi\hookleftarrow\top)\land ([\top\hookrightarrow\varphi]\hookleftarrow\top),
    \end{align*}    
    and write $Sxy$, or $xS y$, if there are $a,b$ such that
    \begin{itemize}
        \item $Rxay$; or
        \item $Rxya$; or
        \item $Rx(ay)b$.\footnote{Recall Definition~\ref{def:associativeframe}: $Rx(ay)b$ is short for $\exists z\in X (Rxzb \text{ and } Rzay)$.
        }
    \end{itemize} 
    Observe that 
    \begin{align*}
        &\mathbb{M}, x\Vdash \Box\varphi && \text{iff} && \text{for all $y\in X$ such that } xS y: \mathbb{M}, y\Vdash \varphi,
    \end{align*}
    and note that, due to associativity, the relation $S$ is \textit{transitive}: if $xS yS z$, then $xS z$.
    
    For intuition, consider the case where $(X,R)$ is a semilattice; i.e., $Rxyz$ iff $x=y\cdot z$ for $\cdot:X\times X\to X$ a function that is
    \begin{itemize}\itemsep.2cm
        \item commutative: $x\cdot y=y\cdot x$,
        \item associative: $(x\cdot y)\cdot z=x\cdot(y\cdot z)$,
        \item idempotent: $x\cdot x=x$.
    \end{itemize}
    Then $S\subseteq X^2$ is simply the partial order $\geq$ on $X$ defined by the semilattice operation,
    \begin{align*}
        &x\geq y && \text{:iff} && x=x\cdot y,\footnotemark
    \intertext{and $\Box$ the `past-looking' boxed modality}
        &\mathbb{M}, x\Vdash \Box\varphi&&  \text{iff}&& \text{for all $y\leq x:$ $\mathbb{M}, y\Vdash \varphi$.}
    \end{align*}
\end{Remark}\footnotetext{Here, we have construed `$\cdot$' as a join-semilattice, though it could just as well have been considered a meet-semilattice. We opted for the former to align with our later focus on $(\mathcal{P}(\mathbb{N}), \cup)$, which, perhaps, is more immediately viewed as a join-semilattice.}    
With this, we now define the tiling formulas.
\begin{Definition}[Tiling formulas]   
    Let $\mathcal{W}$ be a finite set of tiles. Recall, cf. Definition~\ref{def:WangTiles}, that for a tile $t$, its edge colours are denoted by \( t_u \) (up), \( t_d \) (down), \( t_l \) (left), and \( t_r \) (right). Given a tile $t\in\mathcal{W}$, denote the sets of tiles that match to the right and upward, respectively, by
    \[
        R(t)=\{t'\in \mathcal{W}\mid t_r=t'_l\}\quad \text{and} \quad U(t)=\{t'\in \mathcal{W}\mid t_u=t'_d\}.
    \]
For each tile $t\in \mathcal{W}$, we introduce a propositional letter $\mathtt{t}$ and, by an abuse of notation, identify the tile $t\in \mathcal{W}$ with the conjunction of literals
    \begin{align}
        t\mathrel{:=}\mathtt{t}\land \bigwedge_{t\neq t'\in \mathcal{W}}\neg \mathtt{t'}. \label{eq:tile}
    \end{align}
    By a further abuse, we define formulas corresponding to the sets of tiles that match to the right and upward, respectively, by
    \begin{align}
            \mathtt{R}(t)\mathrel{:=}\bigvee_{t'\in R(t)}t' \quad \text{and} \quad \mathtt{U}(t)\mathrel{:=}\bigvee_{t'\in U(t)}t'.\label{eq:rightandup}
    \end{align}
    We also introduce propositional letters to encode the quadrant $\mathbb{N}^2$ as well as properties of the coordinate components in pairs $(m,n) \in \mathbb{N}^2$:
    \begin{itemize}
        \item $\mathtt{x_e}$, to encode that the first component $m$ is even, i.e., $m=2m'$.
        \item $\mathtt{x_o}$, to encode that the first component $m$ is odd, i.e., $m=2m'+1$.
        \item $\mathtt{y_e}$, to encode that the second component $n$ is even, i.e., $n=2n'$.
        \item $\mathtt{y_o}$, to encode that the second component $n$ is odd, i.e., $n=2n'+1$.
        \item $\mathtt{x'}$, used in encoding that the first component $m$ is the successor of $m-1$. 
        \item $\mathtt{y'}$, used in encoding that the second component $n$ is the successor of $n-1$. 
    \end{itemize}
Using these propositional letters, we define the formulas:
    \begin{align*}
        &(\alpha_1) && \Box\left(\mathtt{x_e}\to [\mathtt{x'}\circ \mathtt{x_o}]\right)\\
        &(\alpha_2) && \Box\left(\mathtt{x_o}\to [\mathtt{x'}\circ \mathtt{x_e}]\right)\\
        &(\alpha_3) && \Box\left(\mathtt{y_e}\to [\mathtt{y_o}\circ \mathtt{y'}]\right)\\
        &(\alpha_4) && \Box\left(\mathtt{y_o}\to [\mathtt{y_e}\circ \mathtt{y'}]\right)\\
        &(\beta_1) && \Box\left( \Biggl[\bigvee_{\mathtt{a},\mathtt{b}\in \{\mathtt{e},\mathtt{o}\}} \mathtt{x_a} \circ \mathtt{y_b}\Biggr]\to \bigvee_{t\in \mathcal{W}}t\right)\\
        &(\beta_2) && \Box\bigwedge_{\mathtt{a},\mathtt{b}\in \{\mathtt{e},\mathtt{o}\}}\Biggl( \mathtt{x_a} \circ \mathtt{y_b}\to \bigwedge_{(\mathtt{a'},\mathtt{b'})\neq (\mathtt{a},\mathtt{b})}\neg ( \mathtt{x_{a'}}\circ \mathtt{y_{b'}})\Biggr)\\
        &(\gamma_1^h) && \Box\bigwedge_{t\in \mathcal{W}}\biggl(\Bigl[\mathtt{x_e}\circ \mathtt{y_e}\land t\Bigr] \to\Bigl[ \mathtt{x'} \hookrightarrow \bigl(\mathtt{x_e}\circ \mathtt{y_e}\lor[\mathtt{x_o}\circ \mathtt{y_e}\land \mathtt{R}(t)]\bigr)\Bigr]\biggr)\\
        &(\gamma_1^v) && \Box\bigwedge_{t\in \mathcal{W}}\biggl(\Bigl[\mathtt{x_e}\circ \mathtt{y_e}\land t\Bigr] \to\Bigl[ \bigl(\mathtt{x_e}\circ \mathtt{y_e}\lor[\mathtt{x_e}\circ \mathtt{y_o}\land \mathtt{U}(t)]\bigr) \hookleftarrow \mathtt{y'}  \Bigr]\biggr)\\
        &(\gamma_2^h) && \Box\bigwedge_{t\in \mathcal{W}}\biggl(\Bigl[\mathtt{x_e}\circ \mathtt{y_o}\land t\Bigr] \to\Bigl[ \mathtt{x'} \hookrightarrow \bigl(\mathtt{x_e}\circ \mathtt{y_o}\lor[\mathtt{x_o}\circ \mathtt{y_o}\land \mathtt{R}(t)]\bigr)\Bigr]\biggr)\\
        &(\gamma_2^v) && \Box\bigwedge_{t\in \mathcal{W}}\biggl(\Bigl[\mathtt{x_e}\circ \mathtt{y_o}\land t\Bigr] \to\Bigl[ \bigl(\mathtt{x_e}\circ \mathtt{y_o}\lor[\mathtt{x_e}\circ \mathtt{y_e}\land \mathtt{U}(t)]\bigr) \hookleftarrow \mathtt{y'}  \Bigr]\biggr)\\
        &(\gamma_3^h) && \Box\bigwedge_{t\in \mathcal{W}}\biggl(\Bigl[\mathtt{x_o}\circ \mathtt{y_e}\land t\Bigr] \to\Bigl[ \mathtt{x'} \hookrightarrow \bigl(\mathtt{x_o}\circ \mathtt{y_e}\lor[\mathtt{x_e}\circ \mathtt{y_e}\land \mathtt{R}(t)]\bigr)\Bigr]\biggr)\\
        &(\gamma_3^v) && \Box\bigwedge_{t\in \mathcal{W}}\biggl(\Bigl[\mathtt{x_o}\circ \mathtt{y_e}\land t\Bigr] \to\Bigl[ \bigl(\mathtt{x_o}\circ \mathtt{y_e}\lor[\mathtt{x_o}\circ \mathtt{y_o}\land \mathtt{U}(t)]\bigr) \hookleftarrow \mathtt{y'}  \Bigr]\biggr)\\
        &(\gamma_4^h) && \Box\bigwedge_{t\in \mathcal{W}}\biggl(\Bigl[\mathtt{x_o}\circ \mathtt{y_o}\land t\Bigr] \to\Bigl[ \mathtt{x'} \hookrightarrow \bigl(\mathtt{x_o}\circ \mathtt{y_o}\lor[\mathtt{x_e}\circ \mathtt{y_o}\land \mathtt{R}(t)]\bigr)\Bigr]\biggr)\\
        &(\gamma_4^v) && \Box\bigwedge_{t\in \mathcal{W}}\biggl(\Bigl[\mathtt{x_o}\circ \mathtt{y_o}\land t\Bigr] \to\Bigl[ \bigl(\mathtt{x_o}\circ \mathtt{y_o}\lor[\mathtt{x_o}\circ \mathtt{y_e}\land \mathtt{U}(t)]\bigr) \hookleftarrow \mathtt{y'}  \Bigr]\biggr)
    \end{align*}
    Finally, we let
    \[
        \phi^{}_\mathcal{W} \mathrel{:=} \neg\left(\mathtt{x_e}\circ \mathtt{y_e}\land \bigwedge\alpha_i\land \bigwedge\beta_i\land \bigwedge(\gamma_i^h\land \gamma_i^v)\right).
    \]
\end{Definition}
Clearly, the construction of the tiling formula $\phi^{}_\mathcal{W}$ from a finite set of tiles $\mathcal{W}$ is computable. It therefore only remains to prove the tiling lemmas, Lemma~\ref{lm:tiling} and~\ref{lm:tilingpowerset}, which we now proceed to do. For convenience we restate the lemmas.

\setcounter{savedcounter}{\value{Theorem}}
\setcounter{savedsection}{\value{section}}

\setcounter{section}{3}
\setcounter{Theorem}{2}

\begin{Lemma}[Restated]
    If $\mathcal{W}$ does not tile $\mathbb{N}^2$, then $\phi^{}_\mathcal{W}$ is valid in the least associative normal modal logic (i.e., $\mathsf{BSg}\vDash \phi^{}_\mathcal{W}\approx \top$).
\end{Lemma}

\setcounter{Theorem}{\value{savedcounter}}
\setcounter{section}{\value{savedsection}}

\begin{proof}[Proof of Lemma~\ref{lm:tiling}]
    By contraposition, suppose that $\phi^{}_\mathcal{W}$ is not valid in the least associative normal modal logic, i.e., $\phi^{}_\mathcal{W}\notin\mathbf{AK}_2$ (or algebraically, $\mathsf{BSg}\nvDash \phi^{}_\mathcal{W}\approx \top$). We are then to show that $\mathcal{W}$ tiles $\mathbb{N}^2$.

    By canonicity it follows that there is some associative model $\mathbb{M}=(X,R,V)$ and point $z\in X$ such that
        \[
            \mathbb{M}, z\nVdash \phi^{}_\mathcal{W}, 
        \]
    i.e., 
        \[
            \mathbb{M}, z\Vdash \mathtt{x_e}\circ \mathtt{y_e}\land \bigwedge\alpha_i\land \bigwedge\beta_i\land \bigwedge(\gamma_i^h\land \gamma_i^v).
        \]
Using this, we prove that $\mathcal{W}$ tiles $\mathbb{N}^2$. Before proceeding, we outline the structure of the proof. 

\paragraph{\textit{Outline.}} We will begin by identifying elements of $\mathbb{M}$ 
\[ x_0,x_{1}, x_{2}, \hdots;\quad y_0,y_{1}, y_{2}, \hdots;\quad x_{1}', x_{2}', \hdots; \quad y_{1}', y_{2}', \hdots,
\]
which we, in turn, will use to identify what we will call \textit{staircase points}:
\[ p_{1,1}, p_{2,1}, p_{2,2}, p_{3,2}\hdots, p_{k,k}, p_{k+1,k}, \hdots
\]
These are so named because they form a staircase in that they will satisfy:\footnote{Recall that the infix `$x{R}yz$' is synonymous with the prefix `$Rxyz$'.}
\[
    p_{k,k}\,{R}\, x_{k}\, y_{k} \quad \text{and}\quad  p_{k+1,k}\,{R} \,x_{k+1} \,y_{k},
\]
along with
\[
    p_{k,k}\,{R} \,x_{k+1}' \,p_{k+1,k} \quad \text{and} \quad p_{k+1,k}\,{R} \,p_{k+1,k+1} \,y_{k+1}'.
\]
From the staircase, we find a full grid of elements $p_{m,n} \in \mathbb{M}$ for all $(m,n) \in \mathbb{N}\times \mathbb{N}$. Lastly, by associating each point in the quadrant $(m,n) \in \mathbb{N}\times \mathbb{N}$ with the corresponding grid point $p_{m,n} \in \mathbb{M}$, we can define the tiling $\tau:\mathbb{N}\times \mathbb{N}\to \mathcal{W}$ according to which tile formula $t \in \mathcal{W}$ is satisfied at $p_{m,n}$, i.e., $\mathbb{M},  p_{m,n} \Vdash t$. 
\\\\
Returning to the proof, we begin by locating the points $(x_i)_{i\geq 0},(y_i)_{i\geq 0}$ within $\mathbb{M}$---which, loosely speaking, may be thought of as the first and second axis---along with auxiliary points $(x_i')_{i> 0}, (y_i')_{i>0}$ encoding how $x_{i+1}$ and $y_{i+1}$ are successors of $x_{i}$ and $y_{i}$, respectively. Specifically, we use that $z \Vdash  \mathtt{x_e}\circ \mathtt{y_e}\land \bigwedge \alpha_i$ and show by induction that for all $i\geq 0$: 
    \begin{enumerate}
        \item $z\,{R}\,x_0\,y_0$,
        \item $x_i\,{R}\,x_{i+1}'\, x_{i+1}$,
        \item $y_i\,{R} \,y_{i+1} \,y_{i+1}'$, 
        \item\label{itm:primes} $x_{i+1}'\Vdash \mathtt{x'}$ and $y_{i+1}'\Vdash \mathtt{y'}$,
        \item $zSx_i,\; zSy_i,\; zSx_{i+1}',\; zSy_{i+1}'$,
        \item\label{itm:even} If $i$ is even: $x_i\Vdash \mathtt{x_e}$ and $y_i\Vdash \mathtt{y_e}$,
        \item\label{itm:odd} If $i$ is odd: $x_i\Vdash \mathtt{x_o}$ and $y_i\Vdash \mathtt{y_o}$.
    \end{enumerate}
For the base case, since $z\Vdash \mathtt{x_e}\circ \mathtt{y_e}$, there exist points $x_0,y_0\in \mathbb{M}$ such that  $zRx_0y_0$, $x_0\Vdash \mathtt{x_e}$, and $y_0\Vdash \mathtt{y_e}$.

For the induction step, suppose the induction hypothesis holds for some $n\geq 0$. We give the argument for the case of $n$ even; the case where $n$ is odd is analogous. By induction hypothesis, we have $x_n\Vdash \mathtt{x_e}$, $y_n\Vdash \mathtt{y_e}$, $zSx_n,$ and $zSy_n$. Hence, since $z\Vdash \alpha_1\land \alpha_3$, there exist points $x_{n+1}', x_{n+1}, y_{n+1}, y_{n+1}'\in \mathbb{M}$ such that 
\[
    x_n{R} x_{n+1}' x_{n+1}\quad \text{and}\quad y_n{R} y_{n+1} y_{n+1}'
\]
 as well as
\[
    x_{n+1}'\Vdash \mathtt{x'}, \qquad x_{n+1}\Vdash \mathtt{x_o}, \qquad y_{n+1}\Vdash \mathtt{y_o}, \qquad y_{n+1}'\Vdash \mathtt{y'}.
\]
Thus, items 1--4 and 6--7 are met. For item 5, simply recall that $S$ is transitive in the presence of associativity (cf. Remark~\ref{rm:pastlooking}), thereby completing the induction.
\\\\
Next, also by an induction, we show that there are staircase points 
\[
    p_{1,1}, p_{2, 1}, \hdots, p_{k,k}, p_{k+1,k}, \hdots
\]
such that 
\begin{align*}
    &p_{1,1}\,{R} \,x_{2}' \,p_{2,1}, && p_{2,1}\,{R} \, p_{2,2} \,y_{2}', && p_{2,2}\,{R} \,x_{3}' \,p_{3,2}, && p_{3,2}\,{R} \, p_{3,3} \,y_{3}',\hdots \\
    &p_{k,k}\,{R} \,x_{k+1}' \,p_{k+1,k}, && p_{k+1,k}\,{R} \, p_{k+1,k+1} \,y_{k+1}', &&\hdots,
\end{align*}
while also demonstrating that for all $k\geq 1$, 
\begin{align}
    &p_{k,k}\,{R} \, x_{k} \, y_{k} && \text{and} && p_{k+1,k}\,{R} \,x_{k+1}\, y_{k}, \label{eq:stair1} \\
    &z\,S\,p_{k,k} && \text{and} && z\,S\,p_{k+1,k}. \label{eq:stair2}
\end{align}
For the base case, using that $zRx_0y_0$, $x_0Rx_{1}'x_{1}$, and $y_0Ry_{1} y_{1}'$, it follows by associativity that there is a point $p_{1,1}$ such that
\[
    z\,{R}\, (x_{1}' \,p_{1,1})\,y_{1}' \quad \text{ so }\quad z\,S\,p_{1,1}, \quad \text{and}\quad p_{1,1}\,{R}\,x_{1}\,y_{1}.
\]
Further, $x_1$ witnesses $p_{1,1}R(x_{2}' x_{2})y_1$ because $x_{1}R x_{2}'x_{2}$, hence, by associativity, 
\[
    p_{1,1}\,R\,x_{2}' \,(x_{2}\,y_1),
\]
which means that there is a point $p_{2,1}$ such that
    \[
        p_{1,1}\,R\, x_2' \, p_{2,1} \quad \text{and}\quad p_{2,1}\,R\, x_2 \,y_1,
    \]
As we have already shown that $z\,S\,p_{1,1}$ and the former implies that $p_{1,1}\,S\,p_{2,1}$, transitivity of $S$ gives us $z\,S\,p_{2,1}$, completing the base case. 

The inductive step has two cases depending on whether we assume the claim up to $p_{k,k}\,{R} \,x_{k+1}' \,p_{k+1,k}$ or up to $p_{k+1,k}\,{R} \, p_{k+1,k+1} \,y_{k+1}'$. We treat the former. That is, we must show that there is some $p_{k+1,k+1}$ such that 
\[
    p_{k+1,k}\,{R} \, p_{k+1,k+1} \,y_{k+1}',\quad p_{k+1,k+1}\,{R} \, x_{k+1} \, y_{k+1}, \quad \text{and}\quad  z\,S\,p_{k+1,k+1}. 
\]
By assumption, we have $p_{k+1,k}\, R\,x_{k+1}\,y_k$, so $y_k$ witnesses $p_{k+1,k}\,R\,x_{k+1} (y_{k+1}y_{k+1}')$ because $y_{k}R y_{k+1}y_{k+1}'$, so by associativity,
\[
    p_{k+1,k}\,R\,(x_{k+1}\, y_{k+1})\,y_{k+1}'.
\]
This means that there is a point $p_{k+1,k+1}$ such that
    \[
        p_{k+1,k}\,R\, p_{k+1,k+1} \, y_{k+1}' \quad \text{and}\quad p_{k+1,k+1}\,R\, x_{k+1}\, y_{k+1},
    \]
which---since $z\,S\,p_{k+1,k}\,S\,p_{k+1,k+1}$ implies $z\,S\,p_{k+1,k+1}$---completes the induction and establishes the staircase. 
\\\\
To generate the full grid, we need points $p_{m,n}\in \mathbb{M}$ for all $m,n\geq 1$ such that 
\begin{align}
    z\, S\, p_{m,n}, \quad p_{m,n}\, R\, x_{m+1}'\, p_{m+1, n}, \quad p_{m,n}\, R\, p_{m, n+1}\, y_{n+1}'.\label{eq:gridproperties}
\end{align}
We already have these points along the staircase, hence, by repeatedly applying associativity as noted in Remark~\ref{rm:AssociativityandTiling}, we get the remaining points as illustrated in Figure~\ref{fig:staircase}. To be more precise, the construction can be formalised as a recursion on $\omega$: at stage $k\in\omega$, we construct all points $p_{m,n}\in \mathbb{M}$ with $|m-n|=k$. For $k=0$, these are the diagonal points $p_{m,m}$, which are already given. For $k=1$, the points with $m=n+1$ are staircase points as well; to construct the points with $n=m+1$, we use that $p_{m,m}\,{R} \,x_{n}' \,p_{n,m}$ and $ p_{n,m}\,{R} \, p_{n,n} \,y_{n}'$, whence $p_{m,m}\,R\,x_{n}'\, (p_{n,n}\,y_{n}')$, and so by associativity 
\[
    p_{m,m}\,R\,(x_{n}'\, p_{n,n})\,y_{n}'.
\]
Thus there exists a point $p_{m,n}$ such that
\[
    p_{m,m}\,R\, p_{m,n} \, y_{n}' \quad \text{and}\quad p_{m,n}\,R\, x_{n}'\, p_{n,n},
\]
for which we also have $z\, S\, p_{m,n}$, as $z\,S\,p_{m,m}\,S\,p_{m,n}$. This completes the case where $k=1$. The cases where $k>1$ follow analogously, using the other direction of associativity for when $m>n$.

\begin{figure}[ht]

\begin{center}
\begin{tikzpicture}[>=stealth, node distance=2cm]
  
  \node (a2) at (6,0) {$p_{1,1}$};
  \node (b2) at (8,0) {$p_{2,1}$};
  \node (c2) at (8,2) {$p_{2,2}$};
  \node (b3) at (10,2) {$p_{3,2}$};     \node at (11,2) {$\cdots $};
  \node (c3) at (10,4) {$p_{3,3}$};     \node at (11,4) {$\cdots $};        \node at (10,5) {$\vdots $};
  
  \node (d2) at (6,2) {$p_{1,2}$};
  \node (d3) at (8,4) {$p_{2,3}$};      \node at (8,5) {$\vdots $};
  \node (d4) at (6,4) {$p_{1,3}$};      \node at (6,5) {$\vdots $};
  \node (e2) at (10,0) {$p_{3,1}$};     \node at (11,0) {$\cdots $};


  \draw[->] (a2) -- node[midway,below]{\footnotesize ${_{x_2'}R}$} (b2);
  \draw[->] (b2) -- node[midway,right]{\footnotesize ${R_{y_2'}}$} (c2);
  \draw[->] (c2) -- node[midway,below]{\footnotesize ${_{x_3'}R}$} (b3);
  \draw[->] (b3) -- node[midway,right]{\footnotesize ${R_{y_3'}}$} (c3);

  \draw[->, dotted] (a2) -- node[midway,left]{ \footnotesize $R_{y_2'}$ } (d2);
  \draw[->, dotted] (d2) -- node[midway,below]{\footnotesize ${_{x_2'}R}$} (c2);
    \draw[->, dotted] (d3) -- node[midway,above]{\footnotesize ${_{x_3'}R}$} (c3);
    \draw[->, dotted] (d4) -- node[midway,above]{\footnotesize ${_{x_2'}R}$} (d3);
    \draw[->, dotted] (b2) -- node[midway,below]{\footnotesize ${_{x_3'}R}$} (e2);
    \draw[->, dotted] (d2) -- node[midway,left]{ \footnotesize $R_{y_3'}$ } (d4);
    \draw[->, dotted] (c2) -- node[midway,right]{ \footnotesize $R_{y_3'}$ } (d3);
    \draw[->, dotted] (e2) -- node[midway,right]{ \footnotesize $R_{y_2'}$ } (b3);
\end{tikzpicture}
\caption{Generating $\mathbb{N}^2$ from the staircase: $p_{1,1}, p_{2,1}, p_{2,2}, p_{3,2}, p_{3,3},\hdots$.}
  \label{fig:staircase}
\end{center}
\end{figure}
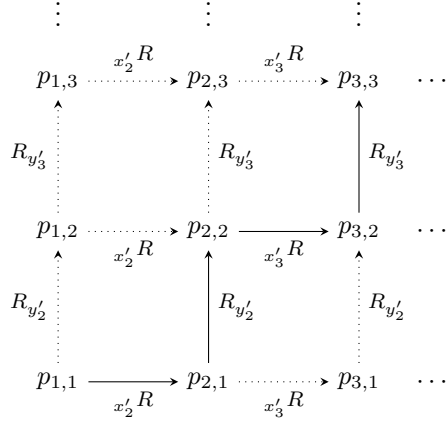


As a consequence of $zSp_{m,n}$ along with $z\Vdash \beta_1$ and the identification in~\eqref{eq:tile} of $t\in\mathcal{W}$ with the conjunction of literals, we have that $\tau:\mathbb{N}^2\to \mathcal{W}$ defined by
    \[
        (m,n)\mapsto  t\in \mathcal{W} \text{ such that }p_{m, n} \Vdash t
    \]
is well-defined as a function, provided that all $p_{m,n}\Vdash \bigvee_{\mathtt{a},\mathtt{b}\in \{\mathtt{e},\mathtt{o}\}} \mathtt{x_a} \circ \mathtt{y_b}$. In fact, to show that $\tau$ defines a tiling, we will need something stronger than this proviso, namely that for all $p_{m,n}$,
 \begin{align*}
     &\text{If $m$ and $n$ are even: } && p_{m,n}\Vdash \mathtt{x_e}\circ \mathtt{y_e}\land \bigwedge_{(\mathtt{a'},\mathtt{b'})\neq (\mathtt{e},\mathtt{e})}\neg ( \mathtt{x_{a'}}\circ \mathtt{y_{b'}}),\\
    &\text{If $m$ is even and $n$ is odd: } && p_{m,n} \Vdash \mathtt{x_e}\circ \mathtt{y_o}\land \bigwedge_{(\mathtt{a'},\mathtt{b'})\neq (\mathtt{e},\mathtt{o})}\neg ( \mathtt{x_{a'}}\circ \mathtt{y_{b'}}),\\
    &\text{If $m$ is odd and $n$ is even: } && p_{m,n} \Vdash \mathtt{x_o}\circ \mathtt{y_e}\land \bigwedge_{(\mathtt{a'},\mathtt{b'})\neq (\mathtt{o},\mathtt{e})}\neg ( \mathtt{x_{a'}}\circ \mathtt{y_{b'}}),\\
    &\text{If $m$ and $n$ are odd: } && p_{m,n} \Vdash \mathtt{x_o}\circ \mathtt{y_o}\land \bigwedge_{(\mathtt{a'},\mathtt{b'})\neq (\mathtt{o},\mathtt{o})}\neg ( \mathtt{x_{a'}}\circ \mathtt{y_{b'}}).
 \end{align*}
We show this by an induction, using the staircase as our base case. For the staircase points, this follows immediately: \eqref{eq:stair1}, together with item~\ref{itm:even} and~\ref{itm:odd}, imply the first conjunct; the second then follows from~\eqref{eq:stair2} and that $z\Vdash \beta_2$.

For the inductive step, we show that if the claim holds for both $p_{m',n'}$ and $p_{m'+1,n'+1}$, then it also holds for $p_{m',n'+1}$ and $p_{m'+1,n'}$. As noted, this will, in particular, entail that $\tau$ is well-defined as a function. To complete the proof, as part of the same argument, we also show that $\tau$ satisfies the tiling conditions: 
\[
    \tau(m',n')_r = \tau(m'+1,n')_l \quad \text{and} \quad \tau(m',n')_u = \tau(m',n'+1)_d.
\]
By our definition of $\tau$ and the formulas $t, \mathtt{R}(t),$ and $\mathtt{U}(t)$ (cf.~\eqref{eq:tile} and~\eqref{eq:rightandup}), this means showing that
\[
    \text{if }p_{m',n'}\Vdash t \text{ then } p_{m'+1,n'}\Vdash \mathtt{R}(t), \quad \text{and} \quad 
    \text{if }p_{m',n'}\Vdash t \text{ then } p_{m',n'+1}\Vdash \mathtt{U}(t).
\]
There are four cases to consider, corresponding to the parities of $m'$ and $n'$. Since the cases are analogous, we only cover the case where $m'$ and $n'$ are even and show that the claim follows for $p_{m',n'+1}$. Assume, then, that the claim holds for $m'=2m$ and $n'=2n$; that is,
\[
    p_{2m,2n}\Vdash \mathtt{x_e}\circ \mathtt{y_e}\land \bigwedge_{(\mathtt{a'},\mathtt{b'})\neq (\mathtt{e},\mathtt{e})}\neg ( \mathtt{x_{a'}}\circ \mathtt{y_{b'}})
\]
and
\[ 
    p_{2m+1,2n+1}\Vdash \mathtt{x_o}\circ \mathtt{y_o}\land \bigwedge_{(\mathtt{a'},\mathtt{b'})\neq (\mathtt{o},\mathtt{o})}\neg ( \mathtt{x_{a'}}\circ \mathtt{y_{b'}}).
\]
Then, as mentioned, $zSp_{2m,2n}$, $p_{2m,2n}\Vdash \mathtt{x_e}\circ \mathtt{y_e}$, and $z\Vdash \beta_1$, jointly imply that $p_{2m,2n}\Vdash t$ for some $t\in \mathcal{W}$. So to complete the proof of this tiling lemma, we show that not only
\[
    p_{2m,2n+1} \Vdash \mathtt{x_e}\circ \mathtt{y_o}\land \bigwedge_{(\mathtt{a'},\mathtt{b'})\neq (\mathtt{e},\mathtt{o})}\neg ( \mathtt{x_{a'}}\circ \mathtt{y_{b'}})
\]
but also
\[
    p_{2m,2n+1} \Vdash\mathtt{U}(t).
\]
Since we have (a) $z\Vdash \gamma_1^v$, (b) $zSp_{2m,2n}$ , (c) $p_{2m,2n}\Vdash \mathtt{x_e}\circ \mathtt{y_e}\land t$, (d) $p_{2m,2n}\,R \,p_{2m, 2n+1}\,y_{2n+1}'$ (by \eqref{eq:gridproperties}), and (e) $y_{2n+1}'\Vdash \mathtt{y'}$ (by item~\ref{itm:primes}), we get that
\[
    p_{2m, 2n+1}\Vdash \mathtt{x_e}\circ \mathtt{y_e}\lor[\mathtt{x_e}\circ \mathtt{y_o}\land \mathtt{U}(t)].
\]
It follows that $p_{2m, 2n+1}\Vdash t'$ for some $t'\in \mathcal{W}$. Hence, if we had $p_{2m, 2n+1}\Vdash \mathtt{x_e}\circ \mathtt{y_e}$, then (a') $z\Vdash \gamma_1^h$, (b') $zSp_{2m,2n+1}$, (c') $p_{2m,2n+1}\Vdash \mathtt{x_e}\circ \mathtt{y_e}\land t'$, (d') $p_{2m,2n+1}\,R \,x_{2m+1}'\,p_{2m+1, 2n+1}$, and (e') $x_{2m+1}'\Vdash \mathtt{x'}$, would jointly imply that
\[
    p_{2m+1, 2n+1}\Vdash \mathtt{x_e}\circ \mathtt{y_e}\lor \mathtt{x_o}\circ \mathtt{y_e},
\]
which is a contradiction, as we have by assumption that
\[
    p_{2m+1,2n+1}\Vdash \mathtt{x_o}\circ \mathtt{y_o}\land \bigwedge_{(\mathtt{a'},\mathtt{b'})\neq (\mathtt{o},\mathtt{o})}\neg ( \mathtt{x_{a'}}\circ \mathtt{y_{b'}}).
\]
Thus, $p_{2m, 2n+1}\nVdash \mathtt{x_e}\circ \mathtt{y_e}$, so $p_{2m, 2n+1}\Vdash \mathtt{x_e}\circ \mathtt{y_o}\land \mathtt{U}(t)$, whence the remaining 
\[
    p_{2m, 2n+1} \Vdash \bigwedge_{(\mathtt{a'},\mathtt{b'})\neq (\mathtt{e},\mathtt{o})}\neg ( \mathtt{x_{a'}}\circ \mathtt{y_{b'}})
\]
follows from $z\Vdash \beta_2$ and $zSp_{2m, 2n+1}$, finalizing our proof.
\end{proof}

We continue with the proof of the other tiling lemma.

\setcounter{savedcounter}{\value{Theorem}}
\setcounter{savedsection}{\value{section}}

\setcounter{section}{3}
\setcounter{Theorem}{3}

\begin{Lemma}[Restated]
    If $\phi^{}_\mathcal{W}$ is valid in $(\mathcal{P}(\mathbb{N}), \cup)$ (i.e., $(\mathcal{P}(\mathbb{N}), \cup)^+\vDash \phi^{}_\mathcal{W}\approx \top$), then $\mathcal{W}$ does not tile $\mathbb{N}^2$.
\end{Lemma}

\setcounter{Theorem}{\value{savedcounter}}
\setcounter{section}{\value{savedsection}}

\begin{proof}[Proof of Lemma~\ref{lm:tilingpowerset}]
    Suppose, contrapositively, that $\tau:\mathbb{N}^2\to \mathcal{W}$ is a tiling. We have to show that $\phi^{}_\mathcal{W}$ is refuted by $(\mathcal{P}(\mathbb{N}), \cup)$, i.e., $\phi^{}_\mathcal{W}\notin \mathrm{Log}(\mathcal{P}(\mathbb{N}),\cup)$ (or algebraically, $(\mathcal{P}(\mathbb{N}), \cup)^+\nvDash \phi^{}_\mathcal{W}\approx \top$).
 
    To this end, write $2\mathbb{N}=\{2n\mid n\in \mathbb{N}\}$ for the even numbers and $2\mathbb{N}+1=\{2n+1\mid n\in \mathbb{N}\}$ for the odd numbers, and define a valuation $V$ as follows:    
\begin{align*}
        V(\mathtt{x_e})&=\{X\subseteq 2\mathbb{N}\mid |2\mathbb{N}\setminus X| \text{ is even}\},\\
        V(\mathtt{x_o})&=\{X\subseteq 2\mathbb{N}\mid |2\mathbb{N}\setminus X| \text{ is odd}\},\\
        V(\mathtt{y_e})&=\{X\subseteq 2\mathbb{N}+1\mid |(2\mathbb{N}+1)\setminus X| \text{ is even}\},\\
        V(\mathtt{y_o})&=\{X\subseteq 2\mathbb{N}+1\mid |(2\mathbb{N}+1)\setminus X| \text{ is odd}\},\\
        V(\mathtt{x'})&=\{\{n\}\mid n\in 2\mathbb{N}\},\\
        V(\mathtt{y'})&=\{\{n\}\mid n\in (2\mathbb{N}+1)\},\\
        V(\mathtt{t})&=\{X\cup Y\mid X\in (V(\mathtt{x_e})\cup V(\mathtt{x_o})), Y\in (V(\mathtt{y_e})\cup V(\mathtt{y_o})),\\
        &\qquad\qquad\qquad\tau(|2\mathbb{N}\setminus X|,|(2\mathbb{N}+1)\setminus Y|)=t\}.
\end{align*}
We proceed showing that
    \begin{align*}
        (\mathcal{P}(\mathbb{N}), \cup, V), \mathbb{N}\nVdash \phi^{}_\mathcal{W},
    \end{align*}
i.e.,
    \begin{align*}
        (\mathcal{P}(\mathbb{N}), \cup, V), \mathbb{N}\Vdash \mathtt{x_e}\circ \mathtt{y_e}\land \bigwedge\alpha_i\land \bigwedge\beta_i\land \bigwedge(\gamma_i^h\land \gamma_i^v).
    \end{align*}
First off, observe that $2\mathbb{N}\Vdash \mathtt{x_e}$ and $2\mathbb{N}+1\Vdash \mathtt{y_e}$, hence
\[
    \mathbb{N}=2\mathbb{N}\cup (2\mathbb{N}+1)\Vdash \mathtt{x_e}\circ \mathtt{y_e}.
\]

\textbf{Alphas ($\boldsymbol{\alpha_i}$).} Second, to see that $\mathbb{N}$ also satisfies the $\alpha$'s, note that if some $X\Vdash \mathtt{x_e}$, then by definition $X\subseteq 2\mathbb{N}$ such that $|2\mathbb{N}\setminus X|$ is even. In particular, there are even numbers $2n\in X$, and for these, we have $\{2n\}\Vdash \mathtt{x'}$ and $|2\mathbb{N}\setminus(X\setminus\{2n\})|$ is odd, hence
    \begin{align*}
        X=\{2n\}\cup (X\setminus \{2n\}) \Vdash \mathtt{x'}\circ \mathtt{x_o},
    \end{align*}
as required. 
A fortiori, $\mathbb{N}\Vdash \alpha_1$. $\mathbb{N}\Vdash \alpha_2\land \alpha_3\land \alpha_4$ is proven analogously.

\textbf{Betas ($\boldsymbol{\beta_i}$).} Note that the decomposition of a $Z\subseteq \mathbb{N}$ into $X\cup Y=Z$ for $X\in (V(\mathtt{x_e})\cup V(\mathtt{x_o})), Y\in (V(\mathtt{y_e})\cup V(\mathtt{y_o}))$ is unique when it exists. So since $\tau:\mathbb{N}^2\to \mathcal{W}$ in particular is a function, we get that $\mathbb{N}\Vdash\beta_1$; and since $V(\mathtt{x_e}), V(\mathtt{x_o}), V(\mathtt{y_e}), V(\mathtt{y_o})$ are pairwise disjoint, we also have $\mathbb{N}\Vdash \beta_2$.

\textbf{Gammas ($\boldsymbol{\gamma_i}$).} In proving that $\mathbb{N}\Vdash\gamma_1^h\land \gamma_1^v\land \cdots \land \gamma_4^h\land \gamma_4^v$, we only cover the first conjunct, as the others are proven analogously. Accordingly, suppose $Z\Vdash \mathtt{x_e}\circ \mathtt{y_e}\land t$ for some tile $t$, and let $Z'\subseteq \mathbb{N}$ and $n\in 2\mathbb{N}$ be arbitrary such that $Z= \{n\}\cup Z'$. For the purpose of demonstrating that $\mathbb{N}\Vdash\gamma_1^h$, we must show that
    \[
        Z'\Vdash \mathtt{x_e}\circ \mathtt{y_e}\lor[\mathtt{x_o}\circ \mathtt{y_e}\land \mathtt{R}(t)].
    \]
If $n\in Z'$, then $Z'=Z\Vdash \mathtt{x_e}\circ \mathtt{y_e}$, as required. \\
If $n\notin Z'$, we use that $Z\Vdash \mathtt{x_e}\circ \mathtt{y_e}$ implies that $Z=X\cup Y$ for some $X\in V(\mathtt{x_e})$ and $Y\in V(\mathtt{y_e})$. Since $Z\ni n$ is even, it follows that $n\in X$, whence
    \[
        Z'=(X\setminus \{n\})\cup Y\Vdash \mathtt{x_o}\circ \mathtt{y_e}.
    \]
Further, using that $Z\Vdash t$, so $X\cup Y=Z\Vdash \mathtt{t}$, which means that $ \tau(|2\mathbb{N}\setminus X|, |(2\mathbb{N}+1)\setminus Y|)=t$, we get that
\begin{align*}
    \tau(|2\mathbb{N}\setminus (X\setminus\{n\})|, |(2\mathbb{N}+1)\setminus Y|)=\tau(|2\mathbb{N}\setminus X|+1, |(2\mathbb{N}+1)\setminus Y|)\in R(t),
\end{align*}
since $\tau$ is a tiling. Consequently, 
\begin{align*}
    Z'=(X\setminus \{n\})\cup Y\Vdash \mathtt{R}(t),
\end{align*}
as desired, completing the proof of $\mathbb{N}\Vdash\gamma_1^h$, and thereby of the lemma.
\end{proof}

\section{Features of the (un)decidable}\label{sec:features}
With the proofs completed, we conclude by turning to a broader aim of this line of research: to better understand the boundary between the solvable and the unsolvable. To this end, we contrast our undecidability results with known decidability results, accentuating features that appear to separate the decidable from the undecidable.

First, the undecidability of the modal logic of semilattices, $\mathrm{Log}(\mathsf{SL})$, (Th. \ref{th:MIL}) contrasts with two established decidability results. As shown in \cite{Knudstorp2023:jpl}, if we expand the class of structures from semilattices to posets (so we do not require existence of all binary suprema), the resulting logic, typically referred to as \textit{modal information logic}, is both decidable and finitely axiomatisable.\footnote{In the literature on modal information logics, `$\circ$' is often denoted by `$\Sup$' (or `$\Inf$'), and its clause is written in Polish notation: $x\Vdash \Sup\varphi\psi$ iff $\exists y,z:y\Vdash \varphi, z\Vdash \psi \text{ and }x=\sup\{y,z\}$.} At first glance, this may appear surprising: why should restricting to semilattices---that is, posets with all binary suprema (or infima)---a seemingly more benign subclass, lead to undecidability? But, as noted by van Benthem \cite{Benthem2024:kripke}, interpreting `$\circ$' via supremum (or infimum) over semilattices rather than posets leads to an important distinction: only over the former is the associativity axiom validated (recall Remark~\ref{rm:AssociativityandTiling} on the role of associativity).

Second, by \cite[Corollary 2]{Knudstorp2023:synthese} (building on \cite{Benthem17:manus, Benthem19:jpl}), truthmaker consequence can be identified with the $\{\lor, \circ\}$-fragment of $\mathrm{Log}(\mathsf{SL})$. This is significant because truthmaker consequence, unlike $\mathrm{Log}(\mathsf{SL})$, is decidable \cite{FineJ19:rsl, Knudstorp2023:synthese}. In fact, as observed in \cite[Remark B.3.8]{Knudstorp2022:thesis}, the decidability proof in \cite{Knudstorp2023:synthese} extends to the language that includes conjunction. Thus, although $\mathrm{Log}(\mathsf{SL})$ is undecidable, its $\{\land,\lor, \circ\}$-fragment, i.e., the negation-free fragment, is decidable.

This inspires the following research question: how much, or how strong a, negation is needed for undecidability? One algebraically formulated precisification of this is whether the equational theory of the variety of \textit{Heyting semigroups}---Heyting algebras with an associative operator distributing over finite joins---is decidable. 
In unpublished work with Peter Jipsen \cite{JipsenKnudstorp:manus}, this was answered affirmatively. Thus, we find a sharp dividing line: unlike the variety of Boolean semigroups, $\mathsf{BSg}$, the variety of Heyting semigroups is decidable. That is, for undecidability, not even intuitionistic negation suffices; classical negation seems necessary. 

Having considered the impact on decidability of altering the \textit{language} or the \textit{frames}, we end with a note on the effect of restricting the \emph{valuations}. As mentioned, \cite{EngstromOlsson2023} conceives hyperboolean modal logic as a unifying framework for propositional team logics and studies it under the name LT. This connection is intriguing because propositional team logics are decidable, yet LT is not (Th. \ref{th:hyp}). The key difference being that valuations in team semantics are principal ideals, while in LT they need not be.

To sharpen this contrast, let $\mathsf{Pow}_\cup$ denote the class of associative frames of the form $(\mathcal{P}(X), \cup)$. By Theorem \ref{th:undecidable}, $\mathrm{Log}(\mathsf{Pow}_\cup)$ is undecidable. However, if we only consider models $(\mathcal{P}(X), \cup, V)$ where $V$ is a principal ideal---in that $V:\mathsf{Prop}\to \{\mathcal{P}(Y)\mid Y\subseteq X\}$---instead of allowing arbitrary valuations $V:\mathsf{Prop}\to \mathcal{P}\mathcal{P}(X)$, we achieve a sound and complete semantics for propositional team logic, which, contra $\mathrm{Log}(\mathsf{Pow}_\cup)$, is decidable.\footnote{Proof of this is included in Appendix~\ref{appendix:teams}.}

\subsection*{Acknowledgements} This research originated during a visit to Chapman University in Fall 2023, where I obtained a first version of the undecidability results presented here. 
I am sincerely grateful to Peter Jipsen for valuable discussions, and for his kindness and encouragement throughout the visit and beyond. I would also like to thank an anonymous referee for a thorough reading and very detailed feedback. My thanks further go to Maria Aloni and Nick Bezhanishvili for comments on the manuscript, and to Johan van Benthem, Cliff Bergman, Valentin Goranko, and Larry Moss for their input on earlier stages of the project.

\subsection*{Funding} The work was partially supported by the MOSAIC project (H2020-MSCA-RISE-2020, 101007627), as well as by the Nothing is Logical (NihiL) project (NWO OC 406.21.CTW.023).

\bibliographystyle{asl}
\bibliography{bibliography}

\appendix
\section{Kripke semantics for propositional team logic}\label{appendix:teams} In this appendix, we prove the concluding claim of Section \ref{sec:features}, namely that a sound and complete Kripke semantics for propositional team logic (PTL) is given by powerset frames $(\mathcal{P}(X), \cup)$ with principal valuations $V:\mathsf{Prop}\to \{\mathcal{P}(Y)\mid Y\subseteq X\}$. Recall that this modal view on team semantics is of interest because PTL is decidable, whereas the logic of powerset frames with arbitrary valuations $\mathrm{Log}(\mathsf{Pow}_\cup)$ is not (by Theorem~\ref{th:undecidable}).

\subsection{Team semantics recap} The language of PTL is given by\footnote{There are not one but many propositional team logics, differing in their choice of connectives and atoms (consult, e.g., \cite{YangVäänänen2017}). However, they typically admit analogous decidability proofs and relational semantics. We focus on the connectives that match the language $\mathcal{L}=\{\land, \lor, \neg,\circ\}$ considered here. See also \cite{EngstromOlsson2023} for a related algebraic perspective on propositional team logics.}
\[        \varphi\mathrel{::=} p \mid \varphi\land \varphi \mid \varphi \lor \varphi \mid   \varphi \vvee \varphi \mid {\sim}\varphi,
\]
where `$\lor$' is the \emph{split} (or \emph{tensor} or \emph{local}) \emph{disjunction}, `$\vvee$' is the \textit{inquisitive} (or \emph{global}) \emph{disjunction}, and `$\sim$' is the \textit{Boolean negation}. 

In team semantics, formulas are not evaluated over single valuations, but over sets of valuations $t\in \mathcal{P}(\{v\mid v:\mathsf{Prop}\to \{0,1\}\})$ (called \textit{teams}):

\begin{align*}
        &t\vDash p && \textbf{iff} && \text{for all } v\in t: v(p)=1,\\
        &t\vDash {\sim}\varphi && \textbf{iff} && t\nvDash \varphi,\\
        &t\vDash \varphi\land\psi && \textbf{iff} && t\vDash \varphi \text{ and } t\vDash \psi,\\
        &t\vDash \varphi\vvee\psi && \textbf{iff} && t\vDash \varphi \text{ or } t\vDash \psi,\\
        &t\vDash \varphi\lor\psi && \textbf{iff} && \text{there exist $t',t''$ such that } t'\vDash \varphi; \\
        & && &&t''\vDash \psi; \text{ and $t=t'\cup t''$.}
    \end{align*}
Decidability follows since (i) formulas $\varphi$ contain only finitely many propositional letters $\mathsf{Prop}(\varphi)$, and (ii) to check whether $t\vDash \varphi$ for all $t\subseteq \{v\mid v:\mathsf{Prop}\to \{0,1\}\}$, it suffices to check whether $t\vDash \varphi$ for all $t\subseteq \{v\mid v:\mathsf{Prop}(\varphi)\to \{0,1\}\}$.

\subsection{From teams to worlds}
For $X=\{v\mid v:\mathsf{Prop}\to \{0,1\}\}$ a set of classical valuations, define the associative frame $(\mathcal{P}(X), \cup)$ with valuation $V$ given by
    \[
        V(p)\mathrel{:=}\{t\in\mathcal{P}(X)\mid \text{for all } v\in t:v(p)=1\}=\mathcal{P}(\{v\in X\mid v(p)=1\}).
    \]
This ensures that
    \[
        (\mathcal{P}(X), \cup, V), t\Vdash p\quad \text{iff}\quad t\in V(p)\quad \text{iff} \quad \text{for all } v\in t:v(p)=1, 
    \]
mirroring the team-semantic atomic clause. Observing that the team-semantic clauses for split disjunction `$\lor$', inquisitive disjunction `$\vvee$', and Boolean negation `$\sim$' precisely are the standard Kripke-semantic clauses for a binary diamond `$\circ$', disjunction `$\lor$', and negation `$\neg$', respectively, it follows that for all $t\in \mathcal{P}(X)$ and $\varphi$,
    \[
        (\mathcal{P}(X), \cup, V),t\Vdash \varphi \quad \text{iff} \quad t\vDash \varphi.
    \]
Here, `$\vDash$' denotes the team-semantical satisfaction relation and `$\Vdash$' the Kripke-semantical satisfaction relation. This yields completeness of the modal semantics for team logic. 

\subsection{From worlds to teams} Conversely, for soundness, given a powerset frame $(\mathcal{P}(X), \cup)$ with principal valuation $V:\mathsf{Prop}\to \{\mathcal{P}(Y)\mid Y\subseteq X\}$, define a classical valuation $v_x:\mathsf{Prop}(\varphi)\to \{0,1\}$ for each $x\in X$ by
        \[
        v_x(p) = 
        \begin{cases}
            1 & \text{if } \{x\} \in V(p), \\
            0 & \text{otherwise.}
        \end{cases}
        \]

\begin{Proposition}
    For all $\varphi$ and $t\in \mathcal{P}(X)$,
    \[
        (\mathcal{P}(X), \cup, V),t\Vdash \varphi \quad \text{iff} \quad \{v_x\mid x\in t\}\vDash \varphi.
    \]
\end{Proposition}
\begin{proof}
    Observe that 
    \[
        t\Vdash p \quad \text{iff} \quad \forall x\in t:\{x\}\Vdash p \quad \text{iff} \quad \forall x\in t:v_x(p)=1\quad \text{iff} \quad \{v_x\mid x\in t\}\vDash p,
    \]
    where the first biconditional follows from $V(p)=\mathcal{P}(Y)$ for some $Y\subseteq X$.
    Using this, it is straightforward to see that the map $t\mapsto \{v_x\mid x\in t\}$ constitutes a p-morphism (w.r.t. $\cup$), so the proposition follows.
\end{proof}
This yields soundness, thereby providing the promised modal semantics for PTL in terms of powerset frames with principal valuations. The decidability of PTL thus contrasts with the undecidability of $\mathrm{Log}(\mathsf{Pow}_\cup)$, where valuations need not be principal.

\end{document}